\documentclass[english]{article}
\usepackage[T1]{fontenc}
\usepackage[utf8]{inputenc}
\usepackage{float}
\usepackage{units}
\usepackage{enumitem}
\usepackage{amsmath}
\usepackage{amsthm}
\usepackage{amssymb}
\usepackage[authoryear]{natbib}

\makeatletter

\newcommand{\lyxmathsym}[1]{\ifmmode\begingroup\def\b@ld{bold}
  \text{\ifx\math@version\b@ld\bfseries\fi#1}\endgroup\else#1\fi}

\floatstyle{ruled}
\newfloat{algorithm}{tbp}{loa}
\providecommand{\algorithmname}{Algorithm}
\floatname{algorithm}{\protect\algorithmname}

\theoremstyle{plain}
\newtheorem{thm}{\protect\theoremname}[section]
\theoremstyle{plain}
\newtheorem{assumption}[thm]{\protect\assumptionname}
\theoremstyle{remark}
\newtheorem{rem}[thm]{\protect\remarkname}
\theoremstyle{plain}
\newtheorem{lem}[thm]{\protect\lemmaname}
\newlist{casenv}{enumerate}{4}
\setlist[casenv]{leftmargin=*,align=left,widest={iiii}}
\setlist[casenv,1]{label={{\itshape\ \casename} \arabic*.},ref=\arabic*}
\setlist[casenv,2]{label={{\itshape\ \casename} \roman*.},ref=\roman*}
\setlist[casenv,3]{label={{\itshape\ \casename\ \alph*.}},ref=\alph*}
\setlist[casenv,4]{label={{\itshape\ \casename} \arabic*.},ref=\arabic*}

\usepackage{dsfont}
\usepackage{microtype}
\usepackage{graphicx}
\usepackage{subcaption}
\usepackage{booktabs} 

\PassOptionsToPackage{hyphens}{url}
\usepackage{hyperref}

\usepackage[accepted]{icml2026}

\usepackage[capitalize,noabbrev]{cleveref}

\usepackage[textsize=tiny]{todonotes}

\icmltitlerunning{Can Adaptive Gradient Methods Converge under Heavy-Tailed Noise? A Case Study of AdaGrad}

\makeatother

\usepackage{babel}
\providecommand{\assumptionname}{Assumption}
\providecommand{\casename}{Case}
\providecommand{\lemmaname}{Lemma}
\providecommand{\remarkname}{Remark}
\providecommand{\theoremname}{Theorem}

\begin{document}
\global\long\def\E{\mathbb{E}}%
\global\long\def\F{\mathcal{F}}%
\global\long\def\N{\mathbb{N}}%
\global\long\def\P{\mathbb{P}}%
\global\long\def\R{\mathbb{R}}%
\global\long\def\O{\mathcal{O}}%
\global\long\def\d{\mathrm{d}}%
\global\long\def\bc{\boldsymbol{c}}%
\global\long\def\bg{\boldsymbol{g}}%
\global\long\def\bL{\boldsymbol{L}}%
\global\long\def\bu{\boldsymbol{u}}%
\global\long\def\bv{\boldsymbol{v}}%
\global\long\def\bw{\boldsymbol{w}}%
\global\long\def\bx{\boldsymbol{x}}%
\global\long\def\by{\boldsymbol{y}}%
\global\long\def\bz{\boldsymbol{z}}%
\global\long\def\bxi{\boldsymbol{\xi}}%
\global\long\def\bsig{\boldsymbol{\sigma}}%
\global\long\def\bzero{\boldsymbol{0}}%
\global\long\def\sx{x}%
\global\long\def\sy{y}%
\global\long\def\sz{z}%
\global\long\def\poly{\mathrm{poly}}%
\global\long\def\p{p}%
\global\long\def\cp{\bar{\p}}%
\global\long\def\defeq{\triangleq}%
\global\long\def\mydots{\dots}%
\global\long\def\res{K}%
\global\long\def\SGD{\mathtt{SGD}}%
\global\long\def\ClipSGD{\mathtt{Clipped}\text{ }\mathtt{SGD}}%
\global\long\def\NSGD{\mathtt{NSGD}}%
\global\long\def\AdaGrad{\mathtt{AdaGrad}}%
\global\long\def\AdaGradNorm{\mathtt{AdaGrad}\text{-}\mathtt{Norm}}%
\global\long\def\RMSProp{\mathtt{RMSProp}}%
\global\long\def\Adam{\mathtt{Adam}}%
\global\long\def\AdamW{\mathtt{AdamW}}%

\twocolumn[
\icmltitle{Can Adaptive Gradient Methods Converge under Heavy-Tailed Noise?\\A Case Study of AdaGrad}



\icmlsetsymbol{equal}{*}

\begin{icmlauthorlist}
	\icmlauthor{Zijian Liu}{NYU}
\end{icmlauthorlist}

\icmlaffiliation{NYU}{Stern School of Business, New York University}

\icmlcorrespondingauthor{Zijian Liu}{zl3067@stern.nyu.edu}

\icmlkeywords{Machine Learning, ICML}

\vskip 0.3in 
]



\printAffiliationsAndNotice{}  
\begin{abstract}
Many tasks in modern machine learning are observed to involve heavy-tailed
gradient noise during the optimization process. To manage this realistic
and challenging setting, new mechanisms, such as gradient clipping
and gradient normalization, have been introduced to ensure the convergence
of first-order algorithms. However, adaptive gradient methods, a famous
class of modern optimizers that includes popular $\mathtt{Adam}$
and $\mathtt{AdamW}$, often perform well even without any extra operations
mentioned above. It is therefore natural to ask whether adaptive gradient
methods can converge under heavy-tailed noise without any algorithmic
changes. In this work, we take the first step toward answering this
question by investigating a special case, $\mathtt{AdaGrad}$, the
origin of adaptive gradient methods. We provide the first provable
convergence rate for $\mathtt{AdaGrad}$ in non-convex optimization
when the tail index $p$ satisfies $4/3<p\leq2$. Notably, this result
is achieved without requiring any prior knowledge of $p$ and is hence
adaptive to the tail index. In addition, we develop an algorithm-dependent
lower bound, suggesting that the existing minimax rate for heavy-tailed
optimization is not attainable by $\mathtt{AdaGrad}$. Lastly, we
consider $\mathtt{AdaGrad}\text{-}\mathtt{Norm}$, a popular variant
of $\mathtt{AdaGrad}$ in theoretical studies, and show an improved
rate that holds for any $1<p\leq2$ under an extra mild assumption.
\end{abstract}

\section{Introduction}

The heavy-tailed phenomenon has been widely observed in the optimization
process for modern machine learning tasks across various domains \citep{pmlr-v97-simsekli19a,pmlr-v139-garg21b,pmlr-v238-battash24a},
and is particularly prevalent when training attention-based models
\citep{NIPS2017_3f5ee243,NEURIPS2020_b05b57f6,ahn2024linear}. More
concretely, it refers to the gradient noise (i.e., the difference
between the stochastic gradient and the true gradient) having only
a finite $\p$-th moment for some $\p\in\left(1,2\right]$, rather
than satisfying the classical finite-variance condition (i.e., $\p=2$)
commonly adopted in the stochastic optimization literature \citep{doi:10.1137/16M1080173,lan2020first}. 

For first-order methods, two approaches are known to guarantee provable
convergence in non-convex optimization under heavy-tailed noise. One
way is based on gradient clipping (i.e., Clipped Stochastic Gradient
Descent ($\ClipSGD$)), which artificially limits the norm of the
stochastic gradient within a user-specific threshold \citep{NEURIPS2020_b05b57f6,pmlr-v195-liu23c,pmlr-v202-sadiev23a,NEURIPS2023_4c454d34}.
The other kind relies on gradient normalization, as recently discovered
by \citet{pmlr-v258-hubler25a,liu2025nonconvex,JMLR:v26:24-1991},
who show that the normalization mechanism in Normalized Stochastic
Gradient Descent (with Momentum) ($\NSGD\text{(}\text{\ensuremath{\mathtt{M}}}\text{)}$)
\citep{nesterov1984minimization,pmlr-v119-cutkosky20b} can also successfully
tackle heavy-tailed noise. Interestingly, \citet{pmlr-v258-hubler25a,liu2025nonconvex}
also provided a convergence rate achieved without any prior knowledge
of problem-dependent parameters, which is the first in the literature.
In particular, the feature of not requiring any information on the
tail index $\p$ highlights a key advantage of using $\NSGD\text{(}\text{\ensuremath{\mathtt{M}}}\text{)}$.

Despite the progress, an important gap still remains between theory
and practice. Specifically, the above results cannot cover a well-known
class of algorithms widely used in practice, namely, adaptive gradient
methods, whose empirical effectiveness has been repeatedly demonstrated
in the training of neural networks, including large language models.
This class includes $\AdaGrad$, introduced in two pioneering works
\citep{McMahanS10,JMLR:v12:duchi11a}, followed by $\RMSProp$ \citep{tieleman2012lecture},
then more practical algorithms nowadays, such as $\Adam$ \citep{kingma2014adam}
and $\AdamW$ \citep{loshchilov2018decoupled}, along with many further
variants. In other words, the reason for the strong performance of
adaptive gradient methods under heavy-tailed noise remains largely
unclear and warrants further exploration.

To the best of our knowledge, only one recent work by \citet{pmlr-v267-chezhegov25a}
studies $\AdaGrad$- and $\Adam$-based methods under heavy-tailed
noise. However, their work has some limitations that we will discuss
below. 

First and foremost, the main results presented in \citet{pmlr-v267-chezhegov25a}
are established for \textit{delayed} adaptive gradient algorithms
introduced by \citet{pmlr-v89-li19c}, in which the stepsize constructed
at the $t$-th iteration depends only on stochastic gradients up to
time $t-1$, rather than time $t$. It is known that the delayed variant
requires a different style of theoretical analysis, since it makes
the stepsize and the stochastic gradient conditionally independent.
Moreover, this modification is rarely implemented in practice. Consequently,
the applicability of their theoretical guarantees to practical algorithms
is limited, leaving the gap open.

Next, Theorem 3.3 in \citet{pmlr-v267-chezhegov25a} is the only result
in their work not for the delayed setting. However, it still cannot
directly apply to $\AdaGrad$ and $\Adam$, as the algorithms considered
there do not employ the coordinate-wise update and additionally require
gradient clipping. Moreover, from a theoretical perspective, their
Theorem 3.3 has two shortcomings. One is the extra assumption of
boundedness for objective functions, which is stronger than the standard
lower boundedness condition in the non-convex optimization literature,
thereby reducing the generality. The other is requiring the value
of problem-dependent parameters (e.g., the tail index $\p$) as inputs
to ensure convergence, contradicting the original purpose of adaptive
gradient methods.

Therefore, \citet{pmlr-v267-chezhegov25a} cannot fully explain the
empirical success of adaptive gradient methods under heavy-tailed
noise, leaving a large room for further improvement.

Motivated by the above discussion, we are naturally led to the following
question:
\begin{center}
\textit{Can adaptive gradient methods converge under heavy-tailed
noise in non-convex optimization, without any algorithmic modifications,
nonstandard assumptions, or prior knowledge of problem-dependent parameters?}
\par\end{center}

\subsection{Our Contributions}

In this work, we take the first step toward answering the above question
through a case study of $\AdaGrad$, the origin of adaptive gradient
methods, and make the following contributions:
\begin{itemize}
\item In Theorem \ref{thm:main-AdaGrad-ub}, we show that $\AdaGrad$ provably
converges at a rate of $\tilde{\O}(1/T^{\frac{3\p-4}{4\p}})$ after
$T$ iterations under heavy-tailed noise in non-convex optimization,
which is meaningful when $\p\in\left(\nicefrac{4}{3},2\right]$. To
the best of our knowledge, this result provides the first theoretical
justification for the convergence of $\AdaGrad$ in the heavy-tailed
setting, without any algorithmic modifications, nonstandard assumptions,
or prior knowledge of problem-dependent parameters, thereby (partially)
confirming the question asked earlier.
\item In Theorem \ref{thm:main-AdaGrad-lb-1d}, we establish the first algorithm-dependent
lower bound for $\AdaGrad$ in heavy-tailed non-convex optimization,
explicitly capturing the dependence on the input learning rate. Our
result suggests that the existing minimax rate for heavy-tailed non-convex
optimization \citep{NEURIPS2020_b05b57f6,liu2025nonconvex,pmlr-v313-liu26a}
is generally not attainable by $\AdaGrad$. Even in the special case
of $\p=2$, our bound also improves upon the existing lower bound
for $\AdaGrad$ proved by \citet{pmlr-v291-jiang25c}.
\end{itemize}
Moreover, we study $\AdaGradNorm$ \citep{pmlr-v97-ward19a}, a popular
variant of $\AdaGrad$ in theoretical research.
\begin{itemize}
\item In Theorem \ref{thm:main-AdaGradNorm-ub-bounded}, we prove that $\AdaGradNorm$
converges at a rate of $\O(1/T^{\frac{\p-1}{2\p}})$ under the additional
assumption of bounded objectives, as considered in \citet{pmlr-v267-chezhegov25a}.
This rate never becomes vacuous for any $\p\in\left(1,2\right]$ and
also does not require any prior information on problem-dependent parameters.
\item In Theorem \ref{thm:AdaGradNorm-ub}, we further provide an upper
bound of $\tilde{\O}(1/T^{\frac{3\p-4}{4\p}})$ for $\AdaGradNorm$
without the extra boundedness assumption, matching the rate for $\AdaGrad$
in terms of $T$ in the same setting.
\end{itemize}
Finally, in Section \ref{sec:conclusion}, we also discuss the limitations
of our work and outline possible future directions.

\subsection{Related Work}

We first review the literature on adaptive gradient methods.

\paragraph{Adaptive gradient methods.}

The study of adaptive gradient methods traces back to two pioneering
works \citep{McMahanS10,JMLR:v12:duchi11a}, independently introducing
the first adaptive gradient method $\AdaGrad$. Later on, $\RMSProp$,
a combination of $\AdaGrad$ and the mean square estimation technique,
was proposed by \citet{tieleman2012lecture}. By further incorporating
momentum into $\RMSProp$, $\Adam$ was developed in the seminal work
of \citet{kingma2014adam}. Furthermore, \citet{loshchilov2018decoupled}
introduced decoupled weight decay into $\Adam$, resulting in a new
algorithm now known as $\AdamW$. In addition to these methods, numerous
variants exist in the literature, for example, $\mathtt{AMSGrad}$
\citep{j.2018on} and $\mathtt{Adafactor}$ \citep{pmlr-v80-shazeer18a}.

Although many adaptive gradient methods were originally designed to
guarantee sublinear regret in online convex optimization (e.g., $\AdaGrad$
and $\Adam$), they have been observed to perform well across a wide
range of modern machine learning tasks, which are however typically
non-convex. As far as we know, \citet{pmlr-v97-ward19a} established
the first provable rate for $\AdaGradNorm$ (a simple variant of $\AdaGrad$),
serving as a cornerstone of theoretical studies for adaptive gradient
methods in non-convex optimization. Subsequently, a large body of
work has developed comprehensive studies of the convergence theory
for different adaptive gradient methods, including $\AdaGrad$, $\RMSProp$,
$\Adam$, and $\AdamW$ (or their variants), in both deterministic
and stochastic non-convex optimization \citep{NEURIPS2018_90365351,de2018convergence,Zou_2019_CVPR,shi2021rmsprop,defossez2022a,kavis2022high,pmlr-v178-faw22a,NEURIPS2022_b6260ae5,pmlr-v195-faw23a,pmlr-v202-liu23aa,pmlr-v202-attia23a,NEURIPS2023_eb1a323f,NEURIPS2023_7ac19fdc,pmlr-v195-wang23a,10.1145/3637528.3671718,NEURIPS2024_14bb27f6,liu2025adagrad,pmlr-v291-jiang25c,zhang2025convergence,JMLR:v26:24-0523}.
However, under the classical finite-variance assumption (or similar
conditions), the existing theory does not reflect a substantial advantage
of adaptive gradient methods over $\SGD$, as they share the same
convergence rate $\tilde{\O}(1/T^{\frac{1}{4}})$ in terms of the
time horizon $T$ to find a stationary point.

As for lower bounds of adaptive gradient methods, the only result
we are aware of is for $\AdaGrad$\footnote{For $\AdaGrad$, \citet{crawshaw2025complexity} also proved a lower
bound result that is not directly comparable to \citet{pmlr-v291-jiang25c}
due to a different setting.} under $\p=2$ given by \citet{pmlr-v291-jiang25c}, which establishes
a complexity lower bound of $\Omega((\Delta L\epsilon^{-2}+\Delta L\sigma^{2}\epsilon^{-4})\ln(\Delta L\epsilon^{-2}))$\footnote{For convenience, we only state their result in the one-dimensional
case. \citet{pmlr-v291-jiang25c} also established a lower bound in
the high-dimensional case.} to find an $\epsilon$-stationary point, where $\Delta$ denotes
the initial function value gap, $L>0$ is the smoothness parameter,
and $\sigma\geq0$ characterizes the noise level. This bound is larger
than the minimax rate for stochastic non-convex optimization under
the finite-variance condition \citep{arjevani2023lower}. However,
it is not algorithm-dependent, since it fails to capture the dependence
on the input learning rate.\\

Next, we provide a basic background on smooth non-convex optimization
under heavy-tailed noise.

\paragraph{Upper bound under heavy-tailed noise.}

For clipping-based algorithms (e.g., $\ClipSGD$), several works have
established the optimal convergence rate $\O(1/T^{\frac{\p-1}{3\p-2}})$
(or $\tilde{\O}(1/T^{\frac{\p-1}{3\p-2}})$) in both expectation and
high probability \citep{NEURIPS2020_b05b57f6,NEURIPS2021_26901deb,pmlr-v195-liu23c,NEURIPS2023_4c454d34}.
Such rates are always derived based on the prior value of problem-dependent
parameters, in particular, the tail index $\p$. Recent works \citep{pmlr-v258-hubler25a,liu2025nonconvex,JMLR:v26:24-1991}
further show that the normalization-based method (i.e., $\NSGD\text{(}\mathtt{M}\text{)}$)
can also achieve the optimal rate of $\O(1/T^{\frac{\p-1}{3\p-2}})$
when prior information about the problem is available. Moreover, \citet{pmlr-v258-hubler25a,liu2025nonconvex}
also prove that $\NSGD\text{(}\mathtt{M}\text{)}$ can converge at
a rate of $\O(1/T^{\frac{\p-1}{2\p}})$ without any prior information.

\paragraph{Lower bound under heavy-tailed noise.}

For non-convex optimization under heavy-tailed noise, to find an $\epsilon$-stationary
point in expectation, any (possibly randomized) algorithm must query
at least $\Omega(\Delta L\epsilon^{-2}+\Delta L\sigma^{\frac{\p}{\p-1}}\epsilon^{-\frac{3\p-2}{\p-1}})$
stochastic gradients \citep{NEURIPS2020_b05b57f6,liu2025nonconvex,pmlr-v313-liu26a},
where $\p\in\left(1,2\right]$ is the tail index, $\Delta$ denotes
the initial function value gap, $L>0$ is the smoothness parameter,
and $\sigma\geq0$ characterizes the noise level. However, the algorithm
(e.g., $\ClipSGD$) that can achieve this minimax lower bound often
requires prior knowledge of all these parameters, which is usually
not practical.

Recently, \citet{pmlr-v258-hubler25a} establishes the first algorithm-dependent
lower bound for $\NSGD$ in one-dimensional optimization that captures
the dependence on the stepsize and batch size when problem-dependent
parameters are unknown in advance. In particular, their result can
be simplified to an $\Omega((\Delta^{4}+L^{4})\epsilon^{-4}+\sigma^{\frac{2\p}{\p-1}}\epsilon^{-\frac{2\p}{\p-1}})$
lower bound for $\NSGD$ when no prior information is available.

\section{Preliminary}

\begin{algorithm*}[t]
\caption{\label{alg:AdaGrad}$\protect\AdaGrad$ \citep{McMahanS10,JMLR:v12:duchi11a}}

\textbf{Input:} initial point $\bx_{1}\in\R^{d}$, learning rate $\gamma>0$,
hyperparameter $\lambda>0$

\textbf{Initialization:} $\bv_{0}=\bzero$

\textbf{for} $t=1$ \textbf{to} $T$ \textbf{do}

$\quad$$\bv_{t}=\bv_{t-1}+\bg_{t}^{2}$

$\quad$$\bx_{t+1}=\bx_{t}-\frac{\gamma}{\lambda+\sqrt{\bv_{t}}}\bg_{t}$

\textbf{end for}
\end{algorithm*}

\paragraph{Notation.}

In this paper, scalars and vectors are denoted by regular and bold
fonts, respectively. $\N$ is the set of natural numbers (excluding
$0$). $\left[n\right]\triangleq\left\{ 1,\mydots,n\right\} ,\forall n\in\N$.
$\left\lceil \cdot\right\rceil $ is the ceiling function. For $a>1$,
we denote by $\bar{a}$ the conjugate of $a$ (i.e., $\frac{1}{a}+\frac{1}{\bar{a}}=1$).
$\R_{>0}^{d}$ (resp. $\R_{\geq0}^{d}$) is the set of vectors in
$\R^{d}$ whose coordinates are all positive (resp. non-negative).
Given $\boldsymbol{\Lambda}\in\R_{>0}^{d}$, its induced inner product
and norm are $\left\langle \bx,\by\right\rangle _{\boldsymbol{\Lambda}}\triangleq\sum_{i=1}^{d}\bx_{i}\boldsymbol{\Lambda}_{i}\by_{i}$
and $\left\Vert \bx\right\Vert _{\boldsymbol{\Lambda}}\triangleq\sqrt{\left\langle \bx,\bx\right\rangle _{\boldsymbol{\Lambda}}}$,
respectively. In addition, we use the following shorthands $(\bx\by)_{i}\defeq\bx_{i}\by_{i}$,
$\bx^{2}\defeq\bx\bx$, $(\frac{\bx}{\by})_{i}\defeq\frac{\bx_{i}}{\by_{i}}$,
and $(\sqrt{\bz})_{i}=\sqrt{\bz_{i}}$ for any $i\in\left[d\right]$
and $\bx,\by\in\R^{d},\bz\in\R_{\geq0}^{d}$. Given a differentiable
function $h$, $\nabla_{i}h$ is the partial derivative w.r.t. the
$i$-th coordinate.

\paragraph{Objective.}

This work studies the optimization problem
\[
\min_{\bx\in\R^{d}}f(\bx),
\]
where $f:\R^{d}\to\R$ is differentiable and possibly non-convex.
Since finding a global optimal solution can be computationally intractable,
we shift the focus to minimizing $\left\Vert \nabla f(\bx)\right\Vert $
as in the non-convex optimization literature.

\paragraph{Assumptions.}

We first make the following assumptions.
\begin{assumption}[Lower boundedness]
\label{assu:lb}The objective satisfies $f_{\star}\triangleq\inf_{\bx\in\R^{d}}f(\bx)>-\infty$.
\end{assumption}

\begin{assumption}[Smoothness]
\label{assu:smooth}$\exists\bL\in\R_{>0}^{d}$ such that $\left|f(\bx)-f(\by)-\left\langle \nabla f(\by),\bx-\by\right\rangle \right|\leq\frac{1}{2}\left\Vert \bx-\by\right\Vert _{\bL}^{2}$
 for any $\bx,\by\in\R^{d}$, or equivalently, $\left\Vert \nabla f(\bx)-\nabla f(\by)\right\Vert _{\nicefrac{1}{\bL}}\leq\left\Vert \bx-\by\right\Vert _{\bL}$
for any $\bx,\by\in\R^{d}$.
\end{assumption}

Assumption \ref{assu:lb} is standard in the literature. Assumption
\ref{assu:smooth} is the coordinate-wise counterpart of the classical
smoothness condition. This kind of fine-grained smoothness has been
studied before, for example, in \citet{pmlr-v80-bernstein18a,bernstein2018signsgd,liu2023on,pmlr-v291-jiang25c,liu2025adagrad}.

Since we consider stochastic optimization, given a point $\bx_{t}\in\R^{d}$
at the $t$-th iteration, $\bg_{t}$ hereinafter denotes the stochastic
gradient queried at $\bx_{t}$. $\F_{t}\defeq\sigma\left(\bg_{1},\mydots,\bg_{t}\right)$
is the natural filtration, and $\E_{t}\left[\cdot\right]\defeq\E\left[\cdot\mid\F_{t}\right]$
represents the conditional expectation given $\F_{t}$.

Our analysis also relies on the next two assumptions.
\begin{assumption}[Unbiased gradient]
\label{assu:unbias}The stochastic gradient satisfies $\E_{t-1}\left[\bg_{t}\right]=\nabla f(\bx_{t})$.
\end{assumption}

\begin{assumption}[Heavy-tailed noise]
\label{assu:heavy}$\exists\p\in\left(1,2\right]$ and $\bsig\in\R_{\geq0}^{d}$
such that $\E_{t-1}\left[\left|\bxi_{t,i}\right|^{\p}\right]\leq\bsig_{i}^{\p},\forall i\in\left[d\right]$,
where $\bxi_{t,i}\triangleq\bg_{t,i}-\nabla_{i}f(\bx_{t})$.
\end{assumption}

\begin{rem}
Our proof strategy still works when replacing Assumption \ref{assu:heavy}
with a weaker version: $\E_{t-1}\left[\left|\bxi_{t,i}\right|^{\p_{i}}\right]\leq\bsig_{i}^{\p_{i}},\forall i\in\left[d\right]$,
where $\p_{i}\in\left(1,2\right]$ can take different values for different
coordinates. However, to make the work more concise, we keep the current
simpler version.
\end{rem}

Assumption \ref{assu:unbias} is a common condition in stochastic
optimization. Assumption \ref{assu:heavy} appears in \citet{pmlr-v267-chezhegov25a}
and differs slightly from the popular one for heavy-tailed noise in
the literature, which typically takes the form of $\E_{t-1}\left[\left\Vert \bxi\right\Vert _{2}^{\p}\right]\leq\sigma^{\p}$
for some $\sigma\geq0$. It can be interpreted as a coordinate-wise
version of heavy-tailed noise, which is natural in our setting since
$\AdaGrad$ employs a coordinate-specific update rule. Moreover, Assumption
\ref{assu:heavy} also generalizes the coordinate-wise finite-variance
assumption considered in prior works \citep{pmlr-v80-bernstein18a,bernstein2018signsgd,pmlr-v291-jiang25c,liu2025adagrad,JMLR:v26:24-0523}.
\begin{rem}
A seemingly similar condition to Assumption \ref{assu:heavy}, but
in fact fundamentally different, is Assumption 2 in \citet{NEURIPS2020_b05b57f6},
where each coordinate of the stochastic gradient $\bg_{t,i}$ is assumed
to have a finite $\p$-th moment, rather than each coordinate of the
noise $\bxi_{t,i}$.
\end{rem}

\section{$\protect\AdaGrad$ under Heavy-Tailed Noise\label{sec:AdaGrad}}

The optimizer focused on in this work, $\AdaGrad$, is described in
Algorithm \ref{alg:AdaGrad}. $\AdaGrad$ was independently introduced
in two pioneering works by \citet{McMahanS10} and \citet{JMLR:v12:duchi11a}.
The key mechanism of $\AdaGrad$ is to dynamically adjust the stepsize,
i.e., $\frac{\gamma}{\lambda+\sqrt{\bv_{t}}}$, based on all stochastic
gradient information up to the current iteration in a coordinate-wise
manner.

\subsection{Upper Bound}

We present the first provable upper bound for $\AdaGrad$ under heavy-tailed
noise in the following Theorem \ref{thm:main-AdaGrad-ub}, the proof
of which is deferred to Appendix \ref{sec:AdaGrad-ub}.
\begin{thm}
\label{thm:main-AdaGrad-ub}Under Assumptions \ref{assu:lb}, \ref{assu:smooth},
\ref{assu:unbias}, and \ref{assu:heavy}, let $\Delta\triangleq f(\bx_{1})-f_{\star}$,
then for any $\gamma>0$ and $\lambda>0$, $\AdaGrad$ (Algorithm
\ref{alg:AdaGrad}) guarantees
\begin{align*}
 & \E\left[\frac{1}{T}\sum_{t=1}^{T}\left\Vert \nabla f(\bx_{t})\right\Vert _{1}\right]\\
\leq & \O\left(\frac{A}{\sqrt{T}}+\frac{C\left\Vert \bsig\right\Vert _{1}}{T^{\frac{\p-1}{\p}}}+\frac{\sqrt{B\left\Vert \bsig\right\Vert _{1}}}{T^{\frac{\p-1}{2\p}}}+\frac{\sqrt{C}\left\Vert \bsig\right\Vert _{1}}{T^{\frac{3\p-4}{4\p}}}\right),
\end{align*}
where $A\defeq d\lambda+\frac{\Delta}{\gamma}+\gamma\left\Vert \bL\right\Vert _{1}\ln\res_{T}$,
$B\defeq\frac{\Delta}{\gamma}+\gamma\left\Vert \bL\right\Vert _{1}\ln\res_{T}$,
$C\defeq\ln^{\frac{1}{\cp}+\frac{1}{2}}\res_{T}$, and $\res_{T}$
is in the order of $\poly(T,\left\Vert \bL\right\Vert _{1},\left\Vert \bL\right\Vert _{\infty},\left\Vert \bsig\right\Vert _{\infty},\left\Vert \nabla f(\bx_{1})\right\Vert _{\infty},\gamma,1/\lambda)$.
\end{thm}

\begin{rem}
For the precise definition of $\res_{T}$, we refer the reader to
Theorem \ref{thm:AdaGrad-ub}.
\end{rem}

To the best of our knowledge, Theorem \ref{thm:main-AdaGrad-ub} provides
the first convergence rate for $\AdaGrad$ under heavy-tailed noise
when the tail index $\p$ lies in the regime $\left(\nicefrac{4}{3},2\right]$.
Remarkably, this result holds under the standard assumptions without
requiring any algorithmic modifications or prior knowledge of problem-dependent
parameters. In contrast, the minimax rate $\Theta(1/T^{\frac{\p-1}{3\p-2}})$
in the literature has been achieved only when problem-dependent parameters,
particularly the value of $\p$, are assumed to be known.

The important feature of Theorem \ref{thm:main-AdaGrad-ub} is its
adaptivity. First, it is adaptive to the tail index $\p$. In other
words, $\AdaGrad$ can automatically adapt to the largest admissible
value of $\p$ without any tuning. In particular, in the boundary
case $\p=2$, Theorem \ref{thm:main-AdaGrad-ub} recovers the well-known
$\tilde{\O}(1/T^{\frac{1}{4}})$ rate of $\AdaGrad$, matching existing
results in the literature \citep{pmlr-v195-wang23a,pmlr-v291-jiang25c}.
Next, the convergence rate is also adaptive to the noise level $\left\Vert \bsig\right\Vert _{1}$,
as in the existing analysis of $\AdaGrad$ under $\p=2$. More precisely,
Theorem \ref{thm:main-AdaGrad-ub} recovers the best possible rate
$\tilde{\O}(1/T^{\frac{1}{2}})$ in the noiseless case when $\bsig=\bzero$
(or $\left\Vert \bsig\right\Vert _{1}$ is sufficiently small to be
negligible). In summary, Theorem \ref{thm:main-AdaGrad-ub} is the
first time demonstrating that $\AdaGrad$ is simultaneously adaptive
to both the tail index $\p$ and the noise level $\left\Vert \bsig\right\Vert _{1}$.

Therefore, we would like to highlight that, as far as we know, this
is the first concrete theoretical evidence supporting the advantage
of adaptive gradient methods over $\SGD$ under heavy-tailed noise.

\paragraph{Novel technique in the analysis.}

Now, we discuss the novel part of our proof. Following the literature
on $\AdaGrad$, our analysis also employs a proxy stepsize proposed
by \citet{pmlr-v97-ward19a} (also known as decorrelated stepsize
\citep{pmlr-v178-faw22a}), i.e., a vector $\bw_{t}\in\R_{\geq0}^{d}$
that is predictable (in other words, $\bw_{t}\in\F_{t-1}$) to approximate
$\bv_{t}$. As far as we know, there are typically two choices for
$\bw_{t}$ in the literature. One is to set $\bw_{t}\defeq\bv_{t-1}+(\nabla f(\bx_{t}))^{2}+\bsig^{2}$
\citep{pmlr-v97-ward19a}. The other is to set $\bw_{t}\defeq\bv_{t-1}$
\citep{pmlr-v195-wang23a}.

The technical contribution of our analysis is to further generalize
the first kind of proxy stepsize. Concretely, we set
\[
\bw_{t}\defeq\bv_{t-1}+(\nabla f(\bx_{t}))^{2}+\bc^{2},
\]
where $\bc\in\R_{\geq0}^{d}$ is a free parameter that can be determined
at the end of the proof. Though this extension seems simple, it is
in fact critical to the analysis. If one simply picks $\bc=\bsig$
as in many existing works (e.g., \citet{pmlr-v291-jiang25c}), the
best possible rate we can derive is in the order of $\tilde{\O}(1/T^{\frac{2\p-3}{2\p}})$,
which is always worse than the rate $\tilde{\O}(1/T^{\frac{3\p-4}{4\p}})$
given in Theorem \ref{thm:main-AdaGrad-ub}. Instead, in our proof,
we pick
\[
\bc_{i}\defeq\frac{\bsig_{i}T^{\frac{1}{2}-\frac{1}{\cp}}}{D_{T,i}^{\frac{1}{2}-\frac{1}{\cp}}},\forall i\in\left[d\right],
\]
where
\[
D_{T,i}\defeq2\ln\left(1+\frac{\bsig_{i}T^{\frac{1}{\p}}+\E\left[\sqrt{\sum_{t=1}^{T}(\nabla_{i}f(\bx_{t}))^{2}}\right]}{\lambda/\sqrt{2}}\right).
\]
This choice turns out to be the key to obtaining the rate $\tilde{\O}(1/T^{\frac{3\p-4}{4\p}})$
stated in Theorem \ref{thm:main-AdaGrad-ub}.

In particular, our choice of $\bc$ degenerates to $\bsig$ when $\p=2$,
since $\frac{1}{2}-\frac{1}{\cp}=0$ due to $\cp=\frac{\p}{\p-1}=2$,
meaning that our $\bw_{t}$ naturally recovers the popular proxy stepsize
$\bw_{t}=\bv_{t-1}+(\nabla f(\bx_{t}))^{2}+\bsig^{2}$ in the classical
finite-variance situation. Thus, our technique is indeed a novel generalization.

For more details of this technique, see Appendix \ref{sec:AdaGrad-ub}.

\subsection{Algorithm-Dependent Lower Bound}

\begin{algorithm*}[t]
\caption{\label{alg:AdaGradNorm}$\protect\AdaGradNorm$ \citep{pmlr-v97-ward19a}}

\textbf{Input:} initial point $\bx_{1}\in\R^{d}$, learning rate $\gamma>0$,
hyperparameter $\lambda>0$

\textbf{Initialization:} $v_{0}=0$

\textbf{for} $t=1$ \textbf{to} $T$ \textbf{do}

$\quad$$v_{t}=v_{t-1}+\left\Vert \bg_{t}\right\Vert _{2}^{2}$

$\quad$$\bx_{t+1}=\bx_{t}-\frac{\gamma}{\lambda+\sqrt{v_{t}}}\bg_{t}$

\textbf{end for}
\end{algorithm*}

In this subsection, we provide the first algorithm-dependent lower
bound for $\AdaGrad$ under heavy-tailed non-convex optimization.
\begin{thm}
\label{thm:main-AdaGrad-lb-1d}Let $d=1$, for any given $\Delta>0$,
$L>0$, $\p\in\left(1,2\right]$, $\sigma\geq0$, $0<\epsilon\leq\sqrt{2\Delta L}$,
$\sx_{1}\in\R$, $\gamma>0$, there exists a function $f:\R\to\R$
associated with a stochastic gradient oracle $g$ satisfying Assumptions
\ref{assu:lb} (and also $f(\sx_{1})-\inf_{\sx\in\R}f(\sx)\leq\Delta$),
\ref{assu:smooth} (with parameter $L$), \ref{assu:unbias}, and
\ref{assu:heavy} (with parameters $\p$ and $\sigma$). Moreover,
if using $\AdaGrad$ (Algorithm \ref{alg:AdaGrad}), with initial
point $\sx_{1}$, learning rate $\gamma$, and hyperparameter $\lambda=0$,
to optimize $f$ by interacting with $g$, one must use at least
\[
\Omega\left(\frac{\frac{\Delta^{2}}{\gamma^{2}}+\gamma^{2}L^{2}\ln^{2}(\frac{\gamma L}{\epsilon})}{\epsilon^{2}}+\frac{\left(\frac{\Delta^{2}}{\gamma^{2}}+\gamma^{2}L^{2}\ln^{2}(\frac{\gamma L}{\epsilon})\right)\sigma^{\frac{\p}{\p-1}}}{\epsilon^{\frac{3\p-2}{\p-1}}}\right)
\]
iterations to make $\E\left[\frac{1}{T}\sum_{t=1}^{T}\left|f'(\sx_{t})\right|\right]\leq\O(\epsilon)$
for small enough $\epsilon$.
\end{thm}

\begin{rem}
The requirement of $\lambda=0$ is only for simplicity. See Theorem
\ref{thm:AdaGrad-lb-1d} for the full version that allows $\lambda\geq0$.
\end{rem}

\begin{rem}
For simplicity, we restrict our attention to the case $d=1$. Following
the proof of Theorem 4 in \citet{pmlr-v291-jiang25c}, Theorem \ref{thm:main-AdaGrad-lb-1d}
can be extended to the high-dimensional setting.
\end{rem}

Theorem \ref{thm:main-AdaGrad-lb-1d} provides the first algorithm-dependent
lower bound for $\AdaGrad$ in the case $d=1$, explicitly capturing
the dependence on the input learning rate.

To better understand Theorem \ref{thm:main-AdaGrad-lb-1d}, let us
first consider $\p=2$, corresponding to the finite-variance setting.
In this case, Theorem \ref{thm:main-AdaGrad-lb-1d} degenerates to
\[
\Omega\left(\frac{\frac{\Delta^{2}}{\gamma^{2}}+\gamma^{2}L^{2}\ln^{2}(\frac{\gamma L}{\epsilon})}{\epsilon^{2}}+\frac{\left(\frac{\Delta^{2}}{\gamma^{2}}+\gamma^{2}L^{2}\ln^{2}(\frac{\gamma L}{\epsilon})\right)\sigma^{2}}{\epsilon^{4}}\right).
\]
We claim that the above bound improves upon the existing one-dimensional
lower bound 
\[
\Omega\left(\frac{\Delta L\ln(\frac{\Delta L}{\epsilon^{2}})}{\epsilon^{2}}+\frac{\Delta L\sigma^{2}\ln(\frac{\Delta L}{\epsilon^{2}})}{\epsilon^{4}}\right)
\]
established by \citet{pmlr-v291-jiang25c} (see their Lemma 16). Indeed,
note that
\begin{align*}
 & \inf_{\gamma>0}\frac{\Delta^{2}}{\gamma^{2}}+\gamma^{2}L^{2}\ln^{2}\left(\frac{\gamma L}{\epsilon}\right)\\
= & \inf_{\eta>0}\frac{\Delta L}{2}\left(\frac{1}{\eta}+\eta\ln^{2}\left(\frac{2\eta\Delta L}{\epsilon^{2}}\right)\right)\geq\frac{\Delta L}{2}\ln\left(\frac{2\Delta L}{\epsilon^{2}}\right),
\end{align*}
where we substitute $\gamma=\sqrt{2\eta\Delta/L}$ in the first step
and apply Lemma \ref{lem:AdaGrad-lb-1d-numeric-2} in the second step.
Therefore, our lower bound for $\AdaGrad$ is strictly more refined
than the best-known one in the literature.

For general $\p\in\left(1,2\right]$, Theorem \ref{thm:main-AdaGrad-lb-1d}
is saying that, without any prior information on $\Delta$ and $L$,
$\AdaGrad$ is impossible to attain the minimax rate $\Omega\left(\frac{\Delta L}{\epsilon^{2}}+\frac{\Delta L\sigma^{\frac{\p}{\p-1}}}{\epsilon^{\frac{3\p-2}{\p-1}}}\right)$
\citep{NEURIPS2020_b05b57f6,liu2025nonconvex,pmlr-v313-liu26a} for
non-convex optimization under heavy-tailed noise. Moreover, even if
$\Delta$ and $L$ are known, Theorem \ref{thm:main-AdaGrad-lb-1d}
indicates that an extra polylogarithmic factor is also unavoidable
for $\AdaGrad$. These two facts together reveal some fundamental
limitations of $\AdaGrad$.

Since the proof is rather technical, we provide only a brief overview
here and refer the interested reader to Appendix \ref{sec:AdaGrad-lb}
for the analysis. Our proof builds upon the framework developed in
\citet{pmlr-v291-jiang25c} and \citet{pmlr-v258-hubler25a}. Concretely,
we show that for a certain stochastic gradient oracle, one can construct
a function parameterized by the learning rate $\gamma$ such that,
with constant probability, $f'(\sx_{t})\geq\Omega(\epsilon)$ for
any $t\in\left[T\right]$ if $T$ is smaller than a threshold that
also depends on $\gamma$. As a consequence, we derive an algorithm-dependent
lower bound that explicitly captures the dependence on the input learning
rate.

Finally, we suspect that our lower bound is not tight in $\epsilon$.
Improving it could be an interesting task that we hope will be addressed
in the future.

\section{$\protect\AdaGradNorm$ Can be Faster, Conditionally}

In this section, we consider a variant of $\AdaGrad$, namely $\AdaGradNorm$
(Algorithm \ref{alg:AdaGradNorm}). Unlike $\AdaGrad$, $\AdaGradNorm$
no longer maintains a coordinate-wise update rule. Instead, its stepsize
is now a scalar and is constructed based on the accumulated squared
norm of the stochastic gradients in history. Although $\AdaGradNorm$
is not widely implemented in practice, it is popular in theoretical
studies due to its simplicity.
\begin{rem}
Our analysis of $\AdaGradNorm$ still applies under the common norm-based
heavy-tailed assumption, i.e., $\E_{t-1}\left[\left\Vert \bxi_{t}\right\Vert _{2}^{\p}\right]\leq\sigma^{\p}$
for some $\sigma\geq0$. However, to avoid introducing more assumptions,
we still consider the coordinate-wise Assumption \ref{assu:heavy}
in the analysis of $\AdaGradNorm$.
\end{rem}

\subsection{A Faster Upper Bound}
\begin{thm}
\label{thm:main-AdaGradNorm-ub-bounded}Under Assumptions \ref{assu:lb},
\ref{assu:smooth}, \ref{assu:unbias}, and \ref{assu:heavy}, suppose
$f^{\star}\defeq\sup_{\bx\in\R^{d}}f(\bx)<+\infty$, let $\Delta_{\star}\triangleq f^{\star}-f_{\star}$,
then for any $\gamma>0$ and $\lambda>0$, $\AdaGradNorm$ (Algorithm
\ref{alg:AdaGradNorm}) guarantees
\[
\E\left[\frac{1}{T}\sum_{t=1}^{T}\left\Vert \nabla f(\bx_{t})\right\Vert _{2}\right]\leq\O\left(\frac{A}{\sqrt{T}}+\frac{\sqrt{B\left\Vert \bsig\right\Vert _{\p}}}{T^{\frac{\p-1}{2\p}}}\right),
\]
where $A\defeq\sqrt{\frac{\lambda\Delta_{\star}}{\gamma}}+\frac{\Delta_{\star}}{\gamma}+\gamma\left\Vert \bL\right\Vert _{\infty}$
and $B\defeq\frac{\Delta_{\star}}{\gamma}+\gamma\left\Vert \bL\right\Vert _{\infty}$.
\end{thm}

\begin{rem}
In fact, we can prove a bound using $\E\left[\frac{1}{T}\sum_{t=1}^{T}\left\Vert \nabla f(\bx_{t})\right\Vert _{2}^{2}\right]$
as a stricter convergence metric (see (\ref{eq:main-AdaGradNorm-ub-bounded-5})
later).
\end{rem}

\begin{rem}
Without additionally assuming an upper boundedness on the objective
function, we can still show that $\AdaGradNorm$ converges at a rate
of $\tilde{\O}(1/T^{\frac{3\p-4}{4\p}})$, in the same order as Theorem
\ref{thm:main-AdaGrad-ub}. The interested reader could refer to Theorem
\ref{thm:AdaGradNorm-ub} in Appendix \ref{sec:AdaGradNorm-ub} for
details.
\end{rem}

The key result in this section is Theorem \ref{thm:main-AdaGradNorm-ub-bounded},
stating a faster rate for $\AdaGradNorm$ under an extra assumption
of bounded objectives, which has also been considered in the existing
literature on adaptive gradient methods (e.g., \citet{NEURIPS2021_ac10ff19,pmlr-v267-chezhegov25a}).

Before moving on, we would like to discuss Theorem \ref{thm:main-AdaGradNorm-ub-bounded}
further. First, it gives a faster rate $\O(1/T^{\frac{\p-1}{2\p}})$
without any extra polylogarithmic terms, which never becomes vacuous
for any $\p\in\left(1,2\right]$, significantly improving upon Theorem
\ref{thm:main-AdaGrad-ub}. Moreover, similar to Theorem \ref{thm:main-AdaGrad-ub},
it still adapts to the tail index $\p$ and the noise level $\left\Vert \bsig\right\Vert _{\p}$\footnote{Note that the change from $\left\Vert \bsig\right\Vert _{1}$ in Theorem
\ref{thm:main-AdaGrad-ub} to $\left\Vert \bsig\right\Vert _{\p}$
in Theorem \ref{thm:main-AdaGradNorm-ub-bounded} is reasonable and
can be expected, since we are using the $2$-norm as the convergence
metric.}, reflecting the power of adaptive gradient methods. More interestingly,
this result perfectly matches the best-known rate $\O(1/T^{\frac{\p-1}{2\p}})$
achieved by $\NSGD\text{(}\mathtt{M}\text{)}$ in the case where problem-dependent
parameters are unknown in advance \citep{pmlr-v258-hubler25a,liu2025nonconvex}.

Finally, we would like to comment that, even under the assumption
of bounded objective functions, it is unclear to us whether $\AdaGrad$
can achieve the rate $\O(1/T^{\frac{\p-1}{2\p}})$, since the proof
strategy of Theorem \ref{thm:main-AdaGradNorm-ub-bounded} is specifically
designed for $\AdaGradNorm$ (see Remark \ref{rem:main-AdaGradNorm-ub-bounded-remark-1}).

\subsection{Theoretical Analysis}

In this subsection, we aim to prove Theorem \ref{thm:main-AdaGradNorm-ub-bounded}.
Our proof is strongly inspired by the recent work of \citet{pmlr-v313-liu26a},
which shows that $\AdaGradNorm$ guarantees an optimal regret bound
in online convex optimization under heavy-tailed noise, even without
knowing $\p$. Although the problems studied are quite different,
one will see that the underlying ideas and proof strategies are closely
related (see Remark \ref{rem:main-AdaGradNorm-ub-bounded-remark-2}).

\begin{proof}[Proof of Theorem \ref{thm:main-AdaGradNorm-ub-bounded}]
In the following proof, let us denote the stepsize in $\AdaGradNorm$
by $\gamma_{t}\defeq\frac{\gamma}{\lambda+\sqrt{v_{t}}},\forall t\in\N$.

We start with Assumption \ref{assu:smooth} and use the update rule
of $\AdaGradNorm$ to obtain
\begin{align*}
f(\bx_{t+1}) & \leq f(\bx_{t})+\left\langle \nabla f(\bx_{t}),\bx_{t+1}-\bx_{t}\right\rangle +\frac{\left\Vert \bx_{t+1}-\bx_{t}\right\Vert _{\bL}^{2}}{2}\\
 & =f(\bx_{t})-\gamma_{t}\left\langle \nabla f(\bx_{t}),\bg_{t}\right\rangle +\frac{\gamma_{t}^{2}\left\Vert \bg_{t}\right\Vert _{\bL}^{2}}{2}\\
 & \leq f(\bx_{t})-\gamma_{t}\left\langle \nabla f(\bx_{t}),\bg_{t}\right\rangle +\frac{\gamma_{t}^{2}\left\Vert \bL\right\Vert _{\infty}\left\Vert \bg_{t}\right\Vert _{2}^{2}}{2}.
\end{align*}
Divide both sides of the above inequality by $\gamma_{t}$ and rearrange
terms to have
\[
\left\langle \nabla f(\bx_{t}),\bg_{t}\right\rangle \leq\frac{f(\bx_{t})-f(\bx_{t+1})}{\gamma_{t}}+\frac{\gamma_{t}\left\Vert \bL\right\Vert _{\infty}\left\Vert \bg_{t}\right\Vert _{2}^{2}}{2},
\]
which implies the following inequality after taking expectations on
both sides (due to Assumption \ref{assu:unbias}),
\begin{align}
 & \E\left[\left\Vert \nabla f(\bx_{t})\right\Vert _{2}^{2}\right]\nonumber \\
\leq & \E\left[\frac{f(\bx_{t})-f(\bx_{t+1})}{\gamma_{t}}\right]+\frac{\left\Vert \bL\right\Vert _{\infty}\E\left[\gamma_{t}\left\Vert \bg_{t}\right\Vert _{2}^{2}\right]}{2}.\label{eq:main-AdaGradNorm-ub-bounded-1}
\end{align}
\begin{rem}
\label{rem:main-AdaGradNorm-ub-bounded-remark-1}The above derivation
cannot be applied to $\AdaGrad$ due to the coordinate-wise stepsize
in it.
\end{rem}

\vspace{-0.1cm}
\begin{rem}
\label{rem:main-AdaGradNorm-ub-bounded-remark-2}The reader familiar
with the online convex optimization literature, and in particular
with the regret analysis of $\AdaGradNorm$, may notice that (\ref{eq:main-AdaGradNorm-ub-bounded-1})
is closely related to the standard inequality characterizing single-step
progress, if recognizing $f(\bx_{t})-f_{\star}$ as a notion of ``distance''.
Therefore, under the boundedness assumption on the objective function,
one may expect that $\E\left[\sum_{t=1}^{T}\left\Vert \nabla f(\bx_{t})\right\Vert _{2}^{2}\right]$
grows sublinearly in $T$.

However, the classical regret bound for $\AdaGradNorm$ is proved
in the deterministic setting (or under finite-variance noise). Thanks
to recent progress by \citet{pmlr-v313-liu26a}, which has proved
that $\AdaGradNorm$ is also robust to heavy-tailed noise and achieves
an optimal regret bound without knowing $\p$. Inspired by \citet{pmlr-v313-liu26a},
we are able to present the following analysis.
\end{rem}

We sum (\ref{eq:main-AdaGradNorm-ub-bounded-1}) from $t=1$ to $T$
to have
\begin{align}
 & \E\left[\sum_{t=1}^{T}\left\Vert \nabla f(\bx_{t})\right\Vert _{2}^{2}\right]\nonumber \\
\leq & \E\left[\sum_{t=1}^{T}\frac{f(\bx_{t})-f(\bx_{t+1})}{\gamma_{t}}\right]+\frac{\left\Vert \bL\right\Vert _{\infty}}{2}\E\left[\sum_{t=1}^{T}\gamma_{t}\left\Vert \bg_{t}\right\Vert _{2}^{2}\right].\label{eq:main-AdaGradNorm-ub-bounded-2}
\end{align}
Observe that
\begin{align}
 & \sum_{t=1}^{T}\frac{f(\bx_{t})-f(\bx_{t+1})}{\gamma_{t}}\nonumber \\
= & \frac{f(\bx_{1})-f_{\star}}{\gamma_{1}}-\frac{f(\bx_{T+1})-f_{\star}}{\gamma_{T}}\nonumber \\
 & +\sum_{t=1}^{T-1}\left(\frac{1}{\gamma_{t+1}}-\frac{1}{\gamma_{t}}\right)\left(f(\bx_{t+1})-f_{\star}\right)\nonumber \\
\leq & \frac{f(\bx_{1})-f_{\star}}{\gamma_{1}}+\sum_{t=1}^{T-1}\left(\frac{1}{\gamma_{t+1}}-\frac{1}{\gamma_{t}}\right)\left(f(\bx_{t+1})-f_{\star}\right)\nonumber \\
\overset{(a)}{\leq} & \frac{\Delta_{\star}}{\gamma_{T}}=\frac{\lambda\Delta_{\star}}{\gamma}+\frac{\Delta_{\star}}{\gamma}\sqrt{v_{T}},\label{eq:main-AdaGradNorm-ub-bounded-3}
\end{align}
where $(a)$ is due to
\begin{eqnarray*}
\frac{1}{\gamma_{t}}\leq\frac{1}{\gamma_{t+1}}, & f\leq f^{\star}, & \Delta_{\star}=f^{\star}-f_{\star}.
\end{eqnarray*}

Moreover, we note that
\begin{align}
\gamma_{t}\left\Vert \bg_{t}\right\Vert _{2}^{2} & =\frac{\gamma\left\Vert \bg_{t}\right\Vert _{2}^{2}}{\lambda+\sqrt{v_{t}}}\leq\frac{\gamma\left\Vert \bg_{t}\right\Vert _{2}^{2}}{\sqrt{\lambda^{2}+v_{t}}}\nonumber \\
 & \leq2\gamma\left(\sqrt{\lambda^{2}+v_{t}}-\sqrt{\lambda^{2}+v_{t-1}}\right)\nonumber \\
\Rightarrow\sum_{t=1}^{T}\gamma_{t}\left\Vert \bg_{t}\right\Vert _{2}^{2} & \leq2\gamma\left(\sqrt{\lambda^{2}+v_{T}}-\lambda\right)\leq2\gamma\sqrt{v_{T}}.\label{eq:main-AdaGradNorm-ub-bounded-4}
\end{align}

To ease the notation, we write $u_{T}=\sum_{t=1}^{T}\left\Vert \nabla f(\bx_{t})\right\Vert _{2}^{2}$.
Plug (\ref{eq:main-AdaGradNorm-ub-bounded-3}) and (\ref{eq:main-AdaGradNorm-ub-bounded-4})
back into (\ref{eq:main-AdaGradNorm-ub-bounded-2}) to obtain
\begin{align*}
 & \E\left[u_{T}\right]\leq\frac{\lambda\Delta_{\star}}{\gamma}+\left(\frac{\Delta_{\star}}{\gamma}+\gamma\left\Vert \bL\right\Vert _{\infty}\right)\E\left[\sqrt{v_{T}}\right]\\
\overset{(b)}{\leq} & \frac{\lambda\Delta_{\star}}{\gamma}+\sqrt{2}\left(\frac{\Delta_{\star}}{\gamma}+\gamma\left\Vert \bL\right\Vert _{\infty}\right)\left(\left\Vert \bsig\right\Vert _{\p}T^{\frac{1}{\p}}+\E\left[\sqrt{u_{T}}\right]\right)\\
\overset{(c)}{\leq} & \frac{\lambda\Delta_{\star}}{\gamma}+\sqrt{2}\left(\frac{\Delta_{\star}}{\gamma}+\gamma\left\Vert \bL\right\Vert _{\infty}\right)\left(\left\Vert \bsig\right\Vert _{\p}T^{\frac{1}{\p}}+\sqrt{\E\left[u_{T}\right]}\right),
\end{align*}
where $(b)$ holds by Lemma \ref{lem:main-AdaGradNorm-v} and $(c)$
is due to H\"{o}lder's inequality. Note that by AM-GM inequality
\begin{align*}
 & \sqrt{2}\left(\frac{\Delta_{\star}}{\gamma}+\gamma\left\Vert \bL\right\Vert _{\infty}\right)\sqrt{\E\left[u_{T}\right]}\\
\leq & \left(\frac{\Delta_{\star}}{\gamma}+\gamma\left\Vert \bL\right\Vert _{\infty}\right)^{2}+\frac{\E\left[u_{T}\right]}{2}.
\end{align*}
Next, we rearrange terms, plug in $u_{T}=\sum_{t=1}^{T}\left\Vert \nabla f(\bx_{t})\right\Vert _{2}^{2}$,
and divide both sides by $T$ to obtain
\begin{equation}
\E\left[\frac{1}{T}\sum_{t=1}^{T}\left\Vert \nabla f(\bx_{t})\right\Vert _{2}^{2}\right]\leq\O\left(\frac{A^{2}}{T}+\frac{B\left\Vert \bsig\right\Vert _{\p}}{T^{\frac{\p-1}{\p}}}\right),\label{eq:main-AdaGradNorm-ub-bounded-5}
\end{equation}
where $A=\sqrt{\frac{\lambda\Delta_{\star}}{\gamma}}+\frac{\Delta_{\star}}{\gamma}+\gamma\left\Vert \bL\right\Vert _{\infty}$
and $B=\frac{\Delta_{\star}}{\gamma}+\gamma\left\Vert \bL\right\Vert _{\infty}$
are defined in the statement of Theorem \ref{thm:main-AdaGradNorm-ub-bounded}.

Finally, we can apply the following inequality to recover Theorem
\ref{thm:main-AdaGradNorm-ub-bounded}, 
\begin{align*}
 & \E\left[\frac{1}{T}\sum_{t=1}^{T}\left\Vert \nabla f(\bx_{t})\right\Vert _{2}\right]\leq\E\left[\sqrt{\frac{1}{T}\sum_{t=1}^{T}\left\Vert \nabla f(\bx_{t})\right\Vert _{2}^{2}}\right]\\
\leq & \sqrt{\E\left[\frac{1}{T}\sum_{t=1}^{T}\left\Vert \nabla f(\bx_{t})\right\Vert _{2}^{2}\right]}\overset{(\ref{eq:main-AdaGradNorm-ub-bounded-5})}{\leq}\O\left(\frac{A}{\sqrt{T}}+\frac{\sqrt{B\left\Vert \bsig\right\Vert _{\p}}}{T^{\frac{\p-1}{2\p}}}\right),
\end{align*}
where the first step holds by the concavity of $\sqrt{\sx}$ and the
second step is due to H\"{o}lder's inequality.
\end{proof}

The above proof relies on the following lemma, which shows that the
term $\E\left[\sqrt{v_{t}}\right]$ can be upper bounded by $\left\Vert \bsig\right\Vert _{\p}t^{\frac{1}{\p}}$
and $\E\left[\sqrt{u_{t}}\right]$ for any $\p\in\left[1,2\right]$.
Essentially the same observation was also made in \citet{pmlr-v313-liu26a}.
\begin{lem}
\label{lem:main-AdaGradNorm-v}Under Assumption \ref{assu:heavy},
for any $t\in\N$, $\AdaGradNorm$ (Algorithm \ref{alg:AdaGradNorm})
guarantees
\[
\E\left[\sqrt{v_{t}}\right]\leq\sqrt{2}\left\Vert \bsig\right\Vert _{\p}t^{\frac{1}{\p}}+\E\left[\sqrt{2u_{t}}\right],
\]
where $u_{t}\triangleq\sum_{s=1}^{t}\left\Vert \nabla f(\bx_{s})\right\Vert _{2}^{2},\forall t\in$$\N$.
\end{lem}

\begin{proof}
By the definition of $v_{t}$, we have
\begin{align*}
\sqrt{v_{t}} & =\sqrt{\sum_{s=1}^{t}\left\Vert \bg_{s}\right\Vert _{2}^{2}}=\sqrt{\sum_{s=1}^{t}\left\Vert \bxi_{s}+\nabla f(\bx_{s})\right\Vert _{2}^{2}}\\
 & \leq\sqrt{2\sum_{s=1}^{t}\left\Vert \bxi_{s}\right\Vert _{2}^{2}+2\sum_{s=1}^{t}\left\Vert \nabla f(\bx_{s})\right\Vert _{2}^{2}}\\
 & \leq\sqrt{2\sum_{s=1}^{t}\left\Vert \bxi_{s}\right\Vert _{2}^{2}}+\sqrt{2u_{t}}\\
 & \leq\sqrt{2}\left(\sum_{s=1}^{t}\sum_{i=1}^{d}\left|\bxi_{s,i}\right|^{\p}\right)^{\frac{1}{\p}}+\sqrt{2u_{t}},
\end{align*}
where the last step is by applying $\left\Vert \cdot\right\Vert _{2}\leq\left\Vert \cdot\right\Vert _{\p}$
twice when $\p\in\left[1,2\right]$, i.e.,
\[
\sqrt{\sum_{s=1}^{t}\left\Vert \bxi_{s}\right\Vert _{2}^{2}}\leq\sqrt{\sum_{s=1}^{t}\left\Vert \bxi_{s}\right\Vert _{\p}^{2}}\leq\left(\sum_{s=1}^{t}\left\Vert \bxi_{s}\right\Vert _{\p}^{\p}\right)^{\frac{1}{\p}}.
\]
Finally, by H\"{o}lder's inequality, we conclude that
\begin{align*}
\E\left[\sqrt{v_{t}}\right] & \leq\sqrt{2}\left(\E\left[\sum_{s=1}^{t}\sum_{i=1}^{d}\left|\bxi_{s,i}\right|^{\p}\right]\right)^{\frac{1}{\p}}+\E\left[\sqrt{2u_{t}}\right]\\
 & \leq\sqrt{2}\left\Vert \bsig\right\Vert _{\p}t^{\frac{1}{\p}}+\E\left[\sqrt{2u_{t}}\right],
\end{align*}
where the last step is due to $\E\left[\left|\bxi_{s,i}\right|^{\p}\right]\leq\bsig_{i}^{\p},\forall s\in\left[t\right],i\in\left[d\right]$
by Assumption \ref{assu:heavy}. 
\end{proof}

\section{Conclusion, Limitations, and Future Work\label{sec:conclusion}}

\paragraph{Conclusion.}

This work makes the first attempt to understand whether $\AdaGrad$,
the origin of adaptive gradient methods, can converge in non-convex
optimization under heavy-tailed noise. We partially address this question
by establishing the first convergence rate for $\AdaGrad$ when the
tail index $\p$ lies in the range $\left(\nicefrac{4}{3},2\right]$.
Importantly, the obtained rate adapts to the tail index and the noise
level simultaneously. Moreover, we derive an algorithm-dependent lower
bound for $\AdaGrad$ in the same setting when $d=1$, suggesting
that the existing minimax rate for heavy-tailed non-convex optimization
is not attainable by $\mathtt{AdaGrad}$. In addition, we show that
$\AdaGradNorm$, a popular variant of $\AdaGrad$ in theoretical studies,
can achieve a faster rate for all $\p\in\left(1,2\right]$ when the
objective function is bounded. We believe these results shed new light
on the empirical success of adaptive gradient methods.

\paragraph{Limitations.}

This study has two main limitations. First, the derived upper bound
for $\AdaGrad$ becomes vacuous when $\p\in\left(1,\nicefrac{4}{3}\right]$.
Second, the algorithm-dependent lower bound does not significantly
improve upon the minimax rate in terms of $\epsilon$. Determining
whether these upper and lower bounds are both loose, or whether one
of them is in fact tight, is an important direction for future research.

\paragraph{Future Work.}

Several promising directions remain open for future investigation.
For example, it is worthwhile to study whether more widely implemented
adaptive gradient methods (e.g., $\Adam$ and $\AdamW$) can converge
under heavy-tailed noise and to characterize their limitations, thereby
helping to demystify their strong performance in practice. Another
direction is to analyze the convergence behavior of adaptive gradient
methods under both heavy-tailed noise and other more realistic assumptions,
such as the generalized smoothness condition proposed by \citet{Zhang2020Why}.

\clearpage

\section*{Acknowledgements}

The author thanks the anonymous reviewers for their valuable feedback.

\section*{Impact Statement}

This paper presents work whose goal is to advance the field of machine
learning. There are many potential societal consequences of our work,
none of which we feel must be specifically highlighted here.

\bibliographystyle{icml2026}
\bibliography{ref}

\begin{thebibliography}{55}
\providecommand{\natexlab}[1]{#1}
\providecommand{\url}[1]{\texttt{#1}}
\expandafter\ifx\csname urlstyle\endcsname\relax
  \providecommand{\doi}[1]{doi: #1}\else
  \providecommand{\doi}{doi: \begingroup \urlstyle{rm}\Url}\fi

\bibitem[Ahn et~al.(2024)Ahn, Cheng, Song, Yun, Jadbabaie, and Sra]{ahn2024linear}
Ahn, K., Cheng, X., Song, M., Yun, C., Jadbabaie, A., and Sra, S.
\newblock Linear attention is (maybe) all you need (to understand transformer optimization).
\newblock In \emph{The Twelfth International Conference on Learning Representations}, 2024.
\newblock URL \url{https://openreview.net/forum?id=0uI5415ry7}.

\bibitem[Arjevani et~al.(2023)Arjevani, Carmon, Duchi, Foster, Srebro, and Woodworth]{arjevani2023lower}
Arjevani, Y., Carmon, Y., Duchi, J.~C., Foster, D.~J., Srebro, N., and Woodworth, B.
\newblock Lower bounds for non-convex stochastic optimization.
\newblock \emph{Mathematical Programming}, 199\penalty0 (1-2):\penalty0 165--214, 2023.

\bibitem[Attia \& Koren(2023)Attia and Koren]{pmlr-v202-attia23a}
Attia, A. and Koren, T.
\newblock {SGD} with {A}da{G}rad stepsizes: Full adaptivity with high probability to unknown parameters, unbounded gradients and affine variance.
\newblock In Krause, A., Brunskill, E., Cho, K., Engelhardt, B., Sabato, S., and Scarlett, J. (eds.), \emph{Proceedings of the 40th International Conference on Machine Learning}, volume 202 of \emph{Proceedings of Machine Learning Research}, pp.\  1147--1171. PMLR, 23--29 Jul 2023.
\newblock URL \url{https://proceedings.mlr.press/v202/attia23a.html}.

\bibitem[Battash et~al.(2024)Battash, Wolf, and Lindenbaum]{pmlr-v238-battash24a}
Battash, B., Wolf, L., and Lindenbaum, O.
\newblock Revisiting the noise model of stochastic gradient descent.
\newblock In Dasgupta, S., Mandt, S., and Li, Y. (eds.), \emph{Proceedings of The 27th International Conference on Artificial Intelligence and Statistics}, volume 238 of \emph{Proceedings of Machine Learning Research}, pp.\  4780--4788. PMLR, 02--04 May 2024.
\newblock URL \url{https://proceedings.mlr.press/v238/battash24a.html}.

\bibitem[Bernstein et~al.(2018)Bernstein, Wang, Azizzadenesheli, and Anandkumar]{pmlr-v80-bernstein18a}
Bernstein, J., Wang, Y.-X., Azizzadenesheli, K., and Anandkumar, A.
\newblock sign{SGD}: Compressed optimisation for non-convex problems.
\newblock In Dy, J. and Krause, A. (eds.), \emph{Proceedings of the 35th International Conference on Machine Learning}, volume~80 of \emph{Proceedings of Machine Learning Research}, pp.\  560--569. PMLR, 10--15 Jul 2018.
\newblock URL \url{https://proceedings.mlr.press/v80/bernstein18a.html}.

\bibitem[Bernstein et~al.(2019)Bernstein, Zhao, Azizzadenesheli, and Anandkumar]{bernstein2018signsgd}
Bernstein, J., Zhao, J., Azizzadenesheli, K., and Anandkumar, A.
\newblock sign{SGD} with majority vote is communication efficient and fault tolerant.
\newblock In \emph{International Conference on Learning Representations}, 2019.
\newblock URL \url{https://openreview.net/forum?id=BJxhijAcY7}.

\bibitem[Bottou et~al.(2018)Bottou, Curtis, and Nocedal]{doi:10.1137/16M1080173}
Bottou, L., Curtis, F.~E., and Nocedal, J.
\newblock Optimization methods for large-scale machine learning.
\newblock \emph{SIAM Review}, 60\penalty0 (2):\penalty0 223--311, 2018.
\newblock \doi{10.1137/16M1080173}.
\newblock URL \url{https://doi.org/10.1137/16M1080173}.

\bibitem[Chezhegov et~al.(2025)Chezhegov, Yaroslav, Semenov, Beznosikov, Gasnikov, Horv\'{a}th, Tak\'{a}\v{c}, and Gorbunov]{pmlr-v267-chezhegov25a}
Chezhegov, S., Yaroslav, K., Semenov, A., Beznosikov, A., Gasnikov, A., Horv\'{a}th, S., Tak\'{a}\v{c}, M., and Gorbunov, E.
\newblock Clipping improves {A}dam-norm and {A}da{G}rad-norm when the noise is heavy-tailed.
\newblock In Singh, A., Fazel, M., Hsu, D., Lacoste-Julien, S., Berkenkamp, F., Maharaj, T., Wagstaff, K., and Zhu, J. (eds.), \emph{Proceedings of the 42nd International Conference on Machine Learning}, volume 267 of \emph{Proceedings of Machine Learning Research}, pp.\  10269--10333. PMLR, 13--19 Jul 2025.
\newblock URL \url{https://proceedings.mlr.press/v267/chezhegov25a.html}.

\bibitem[Crawshaw \& Liu(2025)Crawshaw and Liu]{crawshaw2025complexity}
Crawshaw, M. and Liu, M.
\newblock Complexity lower bounds of adaptive gradient algorithms for non-convex stochastic optimization under relaxed smoothness.
\newblock In \emph{The Thirteenth International Conference on Learning Representations}, 2025.
\newblock URL \url{https://openreview.net/forum?id=ZjOXuAfS6l}.

\bibitem[Cutkosky \& Mehta(2020)Cutkosky and Mehta]{pmlr-v119-cutkosky20b}
Cutkosky, A. and Mehta, H.
\newblock Momentum improves normalized {SGD}.
\newblock In III, H.~D. and Singh, A. (eds.), \emph{Proceedings of the 37th International Conference on Machine Learning}, volume 119 of \emph{Proceedings of Machine Learning Research}, pp.\  2260--2268. PMLR, 13--18 Jul 2020.
\newblock URL \url{https://proceedings.mlr.press/v119/cutkosky20b.html}.

\bibitem[Cutkosky \& Mehta(2021)Cutkosky and Mehta]{NEURIPS2021_26901deb}
Cutkosky, A. and Mehta, H.
\newblock High-probability bounds for non-convex stochastic optimization with heavy tails.
\newblock In Ranzato, M., Beygelzimer, A., Dauphin, Y., Liang, P., and Vaughan, J.~W. (eds.), \emph{Advances in Neural Information Processing Systems}, volume~34, pp.\  4883--4895. Curran Associates, Inc., 2021.
\newblock URL \url{https://proceedings.neurips.cc/paper_files/paper/2021/file/26901debb30ea03f0aa833c9de6b81e9-Paper.pdf}.

\bibitem[De et~al.(2018)De, Mukherjee, and Ullah]{de2018convergence}
De, S., Mukherjee, A., and Ullah, E.
\newblock Convergence guarantees for rmsprop and adam in non-convex optimization and an empirical comparison to nesterov acceleration.
\newblock \emph{arXiv preprint arXiv:1807.06766}, 2018.

\bibitem[D{\'e}fossez et~al.(2022)D{\'e}fossez, Bottou, Bach, and Usunier]{defossez2022a}
D{\'e}fossez, A., Bottou, L., Bach, F., and Usunier, N.
\newblock A simple convergence proof of adam and adagrad.
\newblock \emph{Transactions on Machine Learning Research}, 2022.
\newblock ISSN 2835-8856.
\newblock URL \url{https://openreview.net/forum?id=ZPQhzTSWA7}.

\bibitem[Duchi et~al.(2011)Duchi, Hazan, and Singer]{JMLR:v12:duchi11a}
Duchi, J., Hazan, E., and Singer, Y.
\newblock Adaptive subgradient methods for online learning and stochastic optimization.
\newblock \emph{Journal of Machine Learning Research}, 12\penalty0 (61):\penalty0 2121--2159, 2011.
\newblock URL \url{http://jmlr.org/papers/v12/duchi11a.html}.

\bibitem[Faw et~al.(2022)Faw, Tziotis, Caramanis, Mokhtari, Shakkottai, and Ward]{pmlr-v178-faw22a}
Faw, M., Tziotis, I., Caramanis, C., Mokhtari, A., Shakkottai, S., and Ward, R.
\newblock The power of adaptivity in sgd: Self-tuning step sizes with unbounded gradients and affine variance.
\newblock In Loh, P.-L. and Raginsky, M. (eds.), \emph{Proceedings of Thirty Fifth Conference on Learning Theory}, volume 178 of \emph{Proceedings of Machine Learning Research}, pp.\  313--355. PMLR, 02--05 Jul 2022.
\newblock URL \url{https://proceedings.mlr.press/v178/faw22a.html}.

\bibitem[Faw et~al.(2023)Faw, Rout, Caramanis, and Shakkottai]{pmlr-v195-faw23a}
Faw, M., Rout, L., Caramanis, C., and Shakkottai, S.
\newblock Beyond uniform smoothness: A stopped analysis of adaptive sgd.
\newblock In Neu, G. and Rosasco, L. (eds.), \emph{Proceedings of Thirty Sixth Conference on Learning Theory}, volume 195 of \emph{Proceedings of Machine Learning Research}, pp.\  89--160. PMLR, 12--15 Jul 2023.
\newblock URL \url{https://proceedings.mlr.press/v195/faw23a.html}.

\bibitem[Garg et~al.(2021)Garg, Zhanson, Parisotto, Prasad, Kolter, Lipton, Balakrishnan, Salakhutdinov, and Ravikumar]{pmlr-v139-garg21b}
Garg, S., Zhanson, J., Parisotto, E., Prasad, A., Kolter, Z., Lipton, Z., Balakrishnan, S., Salakhutdinov, R., and Ravikumar, P.
\newblock On proximal policy optimization's heavy-tailed gradients.
\newblock In Meila, M. and Zhang, T. (eds.), \emph{Proceedings of the 38th International Conference on Machine Learning}, volume 139 of \emph{Proceedings of Machine Learning Research}, pp.\  3610--3619. PMLR, 18--24 Jul 2021.
\newblock URL \url{https://proceedings.mlr.press/v139/garg21b.html}.

\bibitem[Hong \& Lin(2024)Hong and Lin]{NEURIPS2024_14bb27f6}
Hong, Y. and Lin, J.
\newblock On convergence of adam for stochastic optimization under relaxed assumptions.
\newblock In Globerson, A., Mackey, L., Belgrave, D., Fan, A., Paquet, U., Tomczak, J., and Zhang, C. (eds.), \emph{Advances in Neural Information Processing Systems}, volume~37, pp.\  10827--10877. Curran Associates, Inc., 2024.
\newblock \doi{10.52202/079017-0346}.
\newblock URL \url{https://proceedings.neurips.cc/paper_files/paper/2024/file/14bb27f680bee45d83bc769738e7f9b5-Paper-Conference.pdf}.

\bibitem[H{\"u}bler et~al.(2025)H{\"u}bler, Fatkhullin, and He]{pmlr-v258-hubler25a}
H{\"u}bler, F., Fatkhullin, I., and He, N.
\newblock From gradient clipping to normalization for heavy tailed sgd.
\newblock In Li, Y., Mandt, S., Agrawal, S., and Khan, E. (eds.), \emph{Proceedings of The 28th International Conference on Artificial Intelligence and Statistics}, volume 258 of \emph{Proceedings of Machine Learning Research}, pp.\  2413--2421. PMLR, 03--05 May 2025.
\newblock URL \url{https://proceedings.mlr.press/v258/hubler25a.html}.

\bibitem[Jiang et~al.(2025)Jiang, Maladkar, and Mokhtari]{pmlr-v291-jiang25c}
Jiang, R., Maladkar, D., and Mokhtari, A.
\newblock Provable complexity improvement of adagrad over sgd: Upper and lower bounds in stochastic non-convex optimization.
\newblock In Haghtalab, N. and Moitra, A. (eds.), \emph{Proceedings of Thirty Eighth Conference on Learning Theory}, volume 291 of \emph{Proceedings of Machine Learning Research}, pp.\  3124--3158. PMLR, 30 Jun--04 Jul 2025.
\newblock URL \url{https://proceedings.mlr.press/v291/jiang25c.html}.

\bibitem[Kavis et~al.(2022)Kavis, Levy, and Cevher]{kavis2022high}
Kavis, A., Levy, K.~Y., and Cevher, V.
\newblock High probability bounds for a class of nonconvex algorithms with adagrad stepsize.
\newblock In \emph{International Conference on Learning Representations}, 2022.
\newblock URL \url{https://openreview.net/forum?id=dSw0QtRMJkO}.

\bibitem[Kingma \& Ba(2014)Kingma and Ba]{kingma2014adam}
Kingma, D.~P. and Ba, J.
\newblock Adam: A method for stochastic optimization.
\newblock \emph{arXiv preprint arXiv:1412.6980}, 2014.

\bibitem[Lan(2020)]{lan2020first}
Lan, G.
\newblock \emph{First-order and stochastic optimization methods for machine learning}.
\newblock Springer, 2020.

\bibitem[Levy et~al.(2021)Levy, Kavis, and Cevher]{NEURIPS2021_ac10ff19}
Levy, K., Kavis, A., and Cevher, V.
\newblock Storm+: Fully adaptive sgd with recursive momentum for nonconvex optimization.
\newblock In Ranzato, M., Beygelzimer, A., Dauphin, Y., Liang, P., and Vaughan, J.~W. (eds.), \emph{Advances in Neural Information Processing Systems}, volume~34, pp.\  20571--20582. Curran Associates, Inc., 2021.
\newblock URL \url{https://proceedings.neurips.cc/paper_files/paper/2021/file/ac10ff1941c540cd87c107330996f4f6-Paper.pdf}.

\bibitem[Li et~al.(2025)Li, Dong, and Lin]{JMLR:v26:24-0523}
Li, H., Dong, Y., and Lin, Z.
\newblock On the $o(\sqrt{d}/t^{1/4})$ convergence rate of rmsprop and its momentum extension measured by $\ell_1$ norm.
\newblock \emph{Journal of Machine Learning Research}, 26\penalty0 (131):\penalty0 1--25, 2025.
\newblock URL \url{http://jmlr.org/papers/v26/24-0523.html}.

\bibitem[Li \& Orabona(2019)Li and Orabona]{pmlr-v89-li19c}
Li, X. and Orabona, F.
\newblock On the convergence of stochastic gradient descent with adaptive stepsizes.
\newblock In Chaudhuri, K. and Sugiyama, M. (eds.), \emph{Proceedings of the Twenty-Second International Conference on Artificial Intelligence and Statistics}, volume~89 of \emph{Proceedings of Machine Learning Research}, pp.\  983--992. PMLR, 16--18 Apr 2019.
\newblock URL \url{https://proceedings.mlr.press/v89/li19c.html}.

\bibitem[Liu et~al.(2025)Liu, Pan, and Zhang]{liu2025adagrad}
Liu, Y., Pan, R., and Zhang, T.
\newblock Adagrad under anisotropic smoothness.
\newblock In \emph{The Thirteenth International Conference on Learning Representations}, 2025.
\newblock URL \url{https://openreview.net/forum?id=4GT9uTsAJE}.

\bibitem[Liu(2026)]{pmlr-v313-liu26a}
Liu, Z.
\newblock Online convex optimization with heavy tails: Old algorithms, new regrets, and applications.
\newblock In Telgarsky, M. and Ullman, J. (eds.), \emph{Proceedings of The 37th International Conference on Algorithmic Learning Theory}, volume 313 of \emph{Proceedings of Machine Learning Research}, pp.\  1--47. PMLR, 23--26 Feb 2026.
\newblock URL \url{https://proceedings.mlr.press/v313/liu26a.html}.

\bibitem[Liu \& Zhou(2025)Liu and Zhou]{liu2025nonconvex}
Liu, Z. and Zhou, Z.
\newblock Nonconvex stochastic optimization under heavy-tailed noises: Optimal convergence without gradient clipping.
\newblock In \emph{The Thirteenth International Conference on Learning Representations}, 2025.
\newblock URL \url{https://openreview.net/forum?id=NKotdPUc3L}.

\bibitem[Liu et~al.(2023{\natexlab{a}})Liu, Nguyen, Ene, and Nguyen]{liu2023on}
Liu, Z., Nguyen, T.~D., Ene, A., and Nguyen, H.
\newblock On the convergence of adagrad(norm) on {$\mathbb{R}^{d}$}: Beyond convexity, non-asymptotic rate and acceleration.
\newblock In \emph{The Eleventh International Conference on Learning Representations}, 2023{\natexlab{a}}.
\newblock URL \url{https://openreview.net/forum?id=ULnHxczCBaE}.

\bibitem[Liu et~al.(2023{\natexlab{b}})Liu, Nguyen, Nguyen, Ene, and Nguyen]{pmlr-v202-liu23aa}
Liu, Z., Nguyen, T.~D., Nguyen, T.~H., Ene, A., and Nguyen, H.
\newblock High probability convergence of stochastic gradient methods.
\newblock In Krause, A., Brunskill, E., Cho, K., Engelhardt, B., Sabato, S., and Scarlett, J. (eds.), \emph{Proceedings of the 40th International Conference on Machine Learning}, volume 202 of \emph{Proceedings of Machine Learning Research}, pp.\  21884--21914. PMLR, 23--29 Jul 2023{\natexlab{b}}.
\newblock URL \url{https://proceedings.mlr.press/v202/liu23aa.html}.

\bibitem[Liu et~al.(2023{\natexlab{c}})Liu, Zhang, and Zhou]{pmlr-v195-liu23c}
Liu, Z., Zhang, J., and Zhou, Z.
\newblock Breaking the lower bound with (little) structure: Acceleration in non-convex stochastic optimization with heavy-tailed noise.
\newblock In Neu, G. and Rosasco, L. (eds.), \emph{Proceedings of Thirty Sixth Conference on Learning Theory}, volume 195 of \emph{Proceedings of Machine Learning Research}, pp.\  2266--2290. PMLR, 12--15 Jul 2023{\natexlab{c}}.
\newblock URL \url{https://proceedings.mlr.press/v195/liu23c.html}.

\bibitem[Loshchilov \& Hutter(2019)Loshchilov and Hutter]{loshchilov2018decoupled}
Loshchilov, I. and Hutter, F.
\newblock Decoupled weight decay regularization.
\newblock In \emph{International Conference on Learning Representations}, 2019.
\newblock URL \url{https://openreview.net/forum?id=Bkg6RiCqY7}.

\bibitem[McMahan \& Streeter(2010)McMahan and Streeter]{McMahanS10}
McMahan, H.~B. and Streeter, M.~J.
\newblock Adaptive bound optimization for online convex optimization.
\newblock In \emph{Conference on Learning Theory {(COLT)}}, pp.\  244--256. Omnipress, 2010.

\bibitem[Nesterov(1984)]{nesterov1984minimization}
Nesterov, Y.~E.
\newblock Minimization methods for nonsmooth convex and quasiconvex functions.
\newblock \emph{Matekon}, 29\penalty0 (3):\penalty0 519--531, 1984.

\bibitem[Nguyen et~al.(2023)Nguyen, Nguyen, Ene, and Nguyen]{NEURIPS2023_4c454d34}
Nguyen, T.~D., Nguyen, T.~H., Ene, A., and Nguyen, H.
\newblock Improved convergence in high probability of clipped gradient methods with heavy tailed noise.
\newblock In Oh, A., Naumann, T., Globerson, A., Saenko, K., Hardt, M., and Levine, S. (eds.), \emph{Advances in Neural Information Processing Systems}, volume~36, pp.\  24191--24222. Curran Associates, Inc., 2023.
\newblock URL \url{https://proceedings.neurips.cc/paper_files/paper/2023/file/4c454d34f3a4c8d6b4ca85a918e5d7ba-Paper-Conference.pdf}.

\bibitem[Reddi et~al.(2018)Reddi, Kale, and Kumar]{j.2018on}
Reddi, S.~J., Kale, S., and Kumar, S.
\newblock On the convergence of adam and beyond.
\newblock In \emph{International Conference on Learning Representations}, 2018.
\newblock URL \url{https://openreview.net/forum?id=ryQu7f-RZ}.

\bibitem[Sadiev et~al.(2023)Sadiev, Danilova, Gorbunov, Horv\'{a}th, Gidel, Dvurechensky, Gasnikov, and Richt\'{a}rik]{pmlr-v202-sadiev23a}
Sadiev, A., Danilova, M., Gorbunov, E., Horv\'{a}th, S., Gidel, G., Dvurechensky, P., Gasnikov, A., and Richt\'{a}rik, P.
\newblock High-probability bounds for stochastic optimization and variational inequalities: the case of unbounded variance.
\newblock In Krause, A., Brunskill, E., Cho, K., Engelhardt, B., Sabato, S., and Scarlett, J. (eds.), \emph{Proceedings of the 40th International Conference on Machine Learning}, volume 202 of \emph{Proceedings of Machine Learning Research}, pp.\  29563--29648. PMLR, 23--29 Jul 2023.
\newblock URL \url{https://proceedings.mlr.press/v202/sadiev23a.html}.

\bibitem[Shazeer \& Stern(2018)Shazeer and Stern]{pmlr-v80-shazeer18a}
Shazeer, N. and Stern, M.
\newblock Adafactor: Adaptive learning rates with sublinear memory cost.
\newblock In Dy, J. and Krause, A. (eds.), \emph{Proceedings of the 35th International Conference on Machine Learning}, volume~80 of \emph{Proceedings of Machine Learning Research}, pp.\  4596--4604. PMLR, 10--15 Jul 2018.
\newblock URL \url{https://proceedings.mlr.press/v80/shazeer18a.html}.

\bibitem[Shi et~al.(2021)Shi, Li, Hong, and Sun]{shi2021rmsprop}
Shi, N., Li, D., Hong, M., and Sun, R.
\newblock {RMS}prop converges with proper hyper-parameter.
\newblock In \emph{International Conference on Learning Representations}, 2021.
\newblock URL \url{https://openreview.net/forum?id=3UDSdyIcBDA}.

\bibitem[Simsekli et~al.(2019)Simsekli, Sagun, and Gurbuzbalaban]{pmlr-v97-simsekli19a}
Simsekli, U., Sagun, L., and Gurbuzbalaban, M.
\newblock A tail-index analysis of stochastic gradient noise in deep neural networks.
\newblock In Chaudhuri, K. and Salakhutdinov, R. (eds.), \emph{Proceedings of the 36th International Conference on Machine Learning}, volume~97 of \emph{Proceedings of Machine Learning Research}, pp.\  5827--5837. PMLR, 09--15 Jun 2019.
\newblock URL \url{https://proceedings.mlr.press/v97/simsekli19a.html}.

\bibitem[Sun et~al.(2025)Sun, Liu, and Yuan]{JMLR:v26:24-1991}
Sun, T., Liu, X., and Yuan, K.
\newblock Revisiting gradient normalization and clipping for nonconvex sgd under heavy-tailed noise: Necessity, sufficiency, and acceleration.
\newblock \emph{Journal of Machine Learning Research}, 26\penalty0 (237):\penalty0 1--42, 2025.
\newblock URL \url{http://jmlr.org/papers/v26/24-1991.html}.

\bibitem[Tieleman et~al.(2012)Tieleman, Hinton, et~al.]{tieleman2012lecture}
Tieleman, T., Hinton, G., et~al.
\newblock Lecture 6.5-rmsprop: Divide the gradient by a running average of its recent magnitude.
\newblock \emph{COURSERA: Neural networks for machine learning}, 4\penalty0 (2):\penalty0 26--31, 2012.

\bibitem[Vaswani et~al.(2017)Vaswani, Shazeer, Parmar, Uszkoreit, Jones, Gomez, Kaiser, and Polosukhin]{NIPS2017_3f5ee243}
Vaswani, A., Shazeer, N., Parmar, N., Uszkoreit, J., Jones, L., Gomez, A.~N., Kaiser, L.~u., and Polosukhin, I.
\newblock Attention is all you need.
\newblock In Guyon, I., Luxburg, U.~V., Bengio, S., Wallach, H., Fergus, R., Vishwanathan, S., and Garnett, R. (eds.), \emph{Advances in Neural Information Processing Systems}, volume~30. Curran Associates, Inc., 2017.
\newblock URL \url{https://proceedings.neurips.cc/paper_files/paper/2017/file/3f5ee243547dee91fbd053c1c4a845aa-Paper.pdf}.

\bibitem[Wang et~al.(2023{\natexlab{a}})Wang, Fu, Zhang, Zheng, and Chen]{NEURIPS2023_7ac19fdc}
Wang, B., Fu, J., Zhang, H., Zheng, N., and Chen, W.
\newblock Closing the gap between the upper bound and lower bound of adam\textquotesingle s iteration complexity.
\newblock In Oh, A., Naumann, T., Globerson, A., Saenko, K., Hardt, M., and Levine, S. (eds.), \emph{Advances in Neural Information Processing Systems}, volume~36, pp.\  39006--39032. Curran Associates, Inc., 2023{\natexlab{a}}.
\newblock URL \url{https://proceedings.neurips.cc/paper_files/paper/2023/file/7ac19fdcdf4f311f3e3ef2e7ef4784d7-Paper-Conference.pdf}.

\bibitem[Wang et~al.(2023{\natexlab{b}})Wang, Zhang, Ma, and Chen]{pmlr-v195-wang23a}
Wang, B., Zhang, H., Ma, Z., and Chen, W.
\newblock Convergence of adagrad for non-convex objectives: Simple proofs and relaxed assumptions.
\newblock In Neu, G. and Rosasco, L. (eds.), \emph{Proceedings of Thirty Sixth Conference on Learning Theory}, volume 195 of \emph{Proceedings of Machine Learning Research}, pp.\  161--190. PMLR, 12--15 Jul 2023{\natexlab{b}}.
\newblock URL \url{https://proceedings.mlr.press/v195/wang23a.html}.

\bibitem[Wang et~al.(2024)Wang, Zhang, Zhang, Meng, Sun, Ma, Liu, Luo, and Chen]{10.1145/3637528.3671718}
Wang, B., Zhang, Y., Zhang, H., Meng, Q., Sun, R., Ma, Z.-M., Liu, T.-Y., Luo, Z.-Q., and Chen, W.
\newblock Provable adaptivity of adam under non-uniform smoothness.
\newblock In \emph{Proceedings of the 30th ACM SIGKDD Conference on Knowledge Discovery and Data Mining}, KDD '24, pp.\  2960--2969, New York, NY, USA, 2024. Association for Computing Machinery.
\newblock ISBN 9798400704901.
\newblock \doi{10.1145/3637528.3671718}.
\newblock URL \url{https://doi.org/10.1145/3637528.3671718}.

\bibitem[Ward et~al.(2019)Ward, Wu, and Bottou]{pmlr-v97-ward19a}
Ward, R., Wu, X., and Bottou, L.
\newblock {A}da{G}rad stepsizes: Sharp convergence over nonconvex landscapes.
\newblock In Chaudhuri, K. and Salakhutdinov, R. (eds.), \emph{Proceedings of the 36th International Conference on Machine Learning}, volume~97 of \emph{Proceedings of Machine Learning Research}, pp.\  6677--6686. PMLR, 09--15 Jun 2019.
\newblock URL \url{https://proceedings.mlr.press/v97/ward19a.html}.

\bibitem[YANG et~al.(2023)YANG, Li, Fatkhullin, and He]{NEURIPS2023_eb1a323f}
YANG, J., Li, X., Fatkhullin, I., and He, N.
\newblock Two sides of one coin: the limits of untuned sgd and the power of adaptive methods.
\newblock In Oh, A., Naumann, T., Globerson, A., Saenko, K., Hardt, M., and Levine, S. (eds.), \emph{Advances in Neural Information Processing Systems}, volume~36, pp.\  74257--74288. Curran Associates, Inc., 2023.
\newblock URL \url{https://proceedings.neurips.cc/paper_files/paper/2023/file/eb1a323fa10d4102ff13422476a744ff-Paper-Conference.pdf}.

\bibitem[Zaheer et~al.(2018)Zaheer, Reddi, Sachan, Kale, and Kumar]{NEURIPS2018_90365351}
Zaheer, M., Reddi, S., Sachan, D., Kale, S., and Kumar, S.
\newblock Adaptive methods for nonconvex optimization.
\newblock In Bengio, S., Wallach, H., Larochelle, H., Grauman, K., Cesa-Bianchi, N., and Garnett, R. (eds.), \emph{Advances in Neural Information Processing Systems}, volume~31. Curran Associates, Inc., 2018.
\newblock URL \url{https://proceedings.neurips.cc/paper_files/paper/2018/file/90365351ccc7437a1309dc64e4db32a3-Paper.pdf}.

\bibitem[Zhang et~al.(2020{\natexlab{a}})Zhang, He, Sra, and Jadbabaie]{Zhang2020Why}
Zhang, J., He, T., Sra, S., and Jadbabaie, A.
\newblock Why gradient clipping accelerates training: A theoretical justification for adaptivity.
\newblock In \emph{International Conference on Learning Representations}, 2020{\natexlab{a}}.
\newblock URL \url{https://openreview.net/forum?id=BJgnXpVYwS}.

\bibitem[Zhang et~al.(2020{\natexlab{b}})Zhang, Karimireddy, Veit, Kim, Reddi, Kumar, and Sra]{NEURIPS2020_b05b57f6}
Zhang, J., Karimireddy, S.~P., Veit, A., Kim, S., Reddi, S., Kumar, S., and Sra, S.
\newblock Why are adaptive methods good for attention models?
\newblock In Larochelle, H., Ranzato, M., Hadsell, R., Balcan, M., and Lin, H. (eds.), \emph{Advances in Neural Information Processing Systems}, volume~33, pp.\  15383--15393. Curran Associates, Inc., 2020{\natexlab{b}}.
\newblock URL \url{https://proceedings.neurips.cc/paper_files/paper/2020/file/b05b57f6add810d3b7490866d74c0053-Paper.pdf}.

\bibitem[Zhang et~al.(2025)Zhang, Zhou, and Zou]{zhang2025convergence}
Zhang, Q., Zhou, Y., and Zou, S.
\newblock Convergence guarantees for {RMSP}rop and adam in generalized-smooth non-convex optimization with affine noise variance.
\newblock \emph{Transactions on Machine Learning Research}, 2025.
\newblock ISSN 2835-8856.
\newblock URL \url{https://openreview.net/forum?id=QIzRdjIWnS}.

\bibitem[Zhang et~al.(2022)Zhang, Chen, Shi, Sun, and Luo]{NEURIPS2022_b6260ae5}
Zhang, Y., Chen, C., Shi, N., Sun, R., and Luo, Z.-Q.
\newblock Adam can converge without any modification on update rules.
\newblock In Koyejo, S., Mohamed, S., Agarwal, A., Belgrave, D., Cho, K., and Oh, A. (eds.), \emph{Advances in Neural Information Processing Systems}, volume~35, pp.\  28386--28399. Curran Associates, Inc., 2022.
\newblock URL \url{https://proceedings.neurips.cc/paper_files/paper/2022/file/b6260ae5566442da053e5ab5d691067a-Paper-Conference.pdf}.

\bibitem[Zou et~al.(2019)Zou, Shen, Jie, Zhang, and Liu]{Zou_2019_CVPR}
Zou, F., Shen, L., Jie, Z., Zhang, W., and Liu, W.
\newblock A sufficient condition for convergences of adam and rmsprop.
\newblock In \emph{Proceedings of the IEEE/CVF Conference on Computer Vision and Pattern Recognition (CVPR)}, June 2019.

\end{thebibliography}

\clearpage

\appendix
\onecolumn

\section{Upper Bound for $\protect\AdaGrad$\label{sec:AdaGrad-ub}}

This section provides the full statement of Theorem \ref{thm:main-AdaGrad-ub}
and its proof.

\subsection{Full Theorem and Its Proof}
\begin{thm}[Full statement of Theorem \ref{thm:main-AdaGrad-ub}]
\label{thm:AdaGrad-ub}Under Assumptions \ref{assu:lb}, \ref{assu:smooth},
\ref{assu:unbias}, and \ref{assu:heavy}, let $\Delta\triangleq f(\bx_{1})-f_{\star}$,
then for any $\gamma>0$ and $\lambda>0$, $\AdaGrad$ (Algorithm
\ref{alg:AdaGrad}) guarantees
\begin{align*}
\E\left[\frac{1}{T}\sum_{t=1}^{T}\left\Vert \nabla f(\bx_{t})\right\Vert _{1}\right]\leq & \O\left(\frac{d\lambda+\frac{\Delta}{\gamma}+\gamma\left\Vert \bL\right\Vert _{1}\ln\res_{T}}{\sqrt{T}}+\frac{\left\Vert \bsig\right\Vert _{1}\ln^{\frac{1}{\cp}+\frac{1}{2}}\res_{T}}{T^{\frac{\p-1}{\p}}}\right.\\
 & \left.\quad+\frac{\sqrt{\left(\frac{\Delta}{\gamma}+\gamma\left\Vert \bL\right\Vert _{1}\ln\res_{T}\right)\left\Vert \bsig\right\Vert _{1}}}{T^{\frac{\p-1}{2\p}}}+\frac{\left\Vert \bsig\right\Vert _{1}\ln^{\frac{1}{2\cp}+\frac{1}{4}}\res_{T}}{T^{\frac{3\p-4}{4\p}}}\right),
\end{align*}
where $\res_{T}=1+\frac{\sqrt{2}\left\Vert \bsig\right\Vert _{\infty}T^{\frac{1}{\p}}+2\left\Vert \nabla f(\bx_{1})\right\Vert _{\infty}T^{\frac{1}{2}}+2\gamma\sqrt{\left\Vert \bL\right\Vert _{1}\left\Vert \bL\right\Vert _{\infty}}T^{\frac{3}{2}}}{\lambda}$
is introduced in Lemma \ref{lem:AdaGrad-ln-bound}.
\end{thm}

\begin{proof}
In the following proof, let
\begin{equation}
\bc_{i}\defeq\frac{\bsig_{i}T^{\frac{1}{2}-\frac{1}{\cp}}}{D_{T,i}^{\frac{1}{2}-\frac{1}{\cp}}},\forall i\in\left[d\right],\label{eq:AdaGrad-ub-c}
\end{equation}
where
\begin{equation}
D_{T,i}\defeq2\ln\left(1+\frac{\sqrt{2}\bsig_{i}T^{\frac{1}{\p}}+\E\left[\sqrt{2\bu_{T,i}}\right]}{\lambda}\right),\forall i\in\left[d\right].\label{eq:AdaGrad-ub-D}
\end{equation}
By Lemma \ref{lem:AdaGrad-ln-bound}, we have
\begin{equation}
\E\left[\ln\left(1+\frac{\bv_{T,i}}{\lambda^{2}}\right)\right]\leq D_{T,i}\leq2\ln\res_{T}.\label{eq:AdaGrad-ub-ineq}
\end{equation}

We sum up the inequality in Lemma \ref{lem:AdaGrad-core} (with $\bc$
defined in (\ref{eq:AdaGrad-ub-c})) from $t=1$ to $T$ and use $f(\bx_{1})-f(\bx_{T+1})\leq\Delta$
(Assumption \ref{assu:lb}) to have
\begin{align*}
\frac{\gamma}{2}\E\left[\sum_{t=1}^{T}\left\Vert \frac{(\nabla f(\bx_{t}))^{2}}{\lambda+\sqrt{\bw_{t}}}\right\Vert _{1}\right] & \leq\Delta+\gamma\sum_{i=1}^{d}\frac{\bsig_{i}^{2}}{\bc_{i}}\sum_{t=1}^{T}\left(\E\left[\frac{\bg_{t,i}^{2}}{\lambda^{2}+\bv_{t,i}}\right]\right)^{\frac{2}{\cp}}+\gamma\sum_{i=1}^{d}\left(\bc_{i}+\frac{\gamma\bL_{i}}{2}\right)\sum_{t=1}^{T}\E\left[\frac{\bg_{t,i}^{2}}{\lambda^{2}+\bv_{t,i}}\right]\\
 & \overset{(a)}{\leq}\Delta+\gamma\sum_{i=1}^{d}\frac{\bsig_{i}^{2}}{\bc_{i}}T^{1-\frac{2}{\cp}}\left(\E\left[\ln\left(1+\frac{\bv_{T,i}}{\lambda^{2}}\right)\right]\right)^{\frac{2}{\cp}}+\gamma\sum_{i=1}^{d}\left(\bc_{i}+\frac{\gamma\bL_{i}}{2}\right)\E\left[\ln\left(1+\frac{\bv_{T,i}}{\lambda^{2}}\right)\right]\\
 & \overset{(\ref{eq:AdaGrad-ub-ineq})}{\leq}\Delta+\gamma\sum_{i=1}^{d}\left(\frac{\bsig_{i}^{2}}{\bc_{i}}T^{1-\frac{2}{\cp}}D_{T,i}^{\frac{2}{\cp}}+\bc_{i}D_{T,i}\right)+\gamma^{2}\left\Vert \bL\right\Vert _{1}\ln\res_{T}\\
 & \overset{(\ref{eq:AdaGrad-ub-c})}{=}\Delta+2\gamma\sum_{i=1}^{d}\bsig_{i}T^{\frac{1}{2}-\frac{1}{\cp}}D_{T,i}^{\frac{1}{\cp}+\frac{1}{2}}+\gamma^{2}\left\Vert \bL\right\Vert _{1}\ln\res_{T}\\
 & \overset{(\ref{eq:AdaGrad-ub-ineq})}{\leq}\Delta+2^{\frac{1}{\cp}+\frac{3}{2}}\gamma\left\Vert \bsig\right\Vert _{1}T^{\frac{1}{2}-\frac{1}{\cp}}\ln^{\frac{1}{\cp}+\frac{1}{2}}\res_{T}+\gamma^{2}\left\Vert \bL\right\Vert _{1}\ln\res_{T}\\
 & \overset{(b)}{=}\Delta+4\gamma\left\Vert \bsig\right\Vert _{1}T^{\frac{1}{\p}-\frac{1}{2}}\ln^{\frac{1}{\cp}+\frac{1}{2}}\res_{T}+\gamma^{2}\left\Vert \bL\right\Vert _{1}\ln\res_{T},
\end{align*}
where $(a)$ is by applying Lemma \ref{lem:AdaGrad-general-ln} with
$q=\frac{2}{\cp}$ and $q=1$ and $(b)$ is due to $\frac{1}{\cp}\leq\frac{1}{2}$
and $\frac{1}{2}-\frac{1}{\cp}=\frac{1}{\p}-\frac{1}{2}$. Divide
both sides of the above inequality by $\frac{\gamma}{2}$ to obtain
\begin{equation}
\E\left[\sum_{t=1}^{T}\left\Vert \frac{(\nabla f(\bx_{t}))^{2}}{\lambda+\sqrt{\bw_{t}}}\right\Vert _{1}\right]\leq\frac{2\Delta}{\gamma}+2\gamma\left\Vert \bL\right\Vert _{1}\ln\res_{T}+8\left\Vert \bsig\right\Vert _{1}T^{\frac{1}{\p}-\frac{1}{2}}\ln^{\frac{1}{\cp}+\frac{1}{2}}\res_{T}.\label{eq:AdaGrad-ub-1}
\end{equation}

Next, recall that $\bu_{t,i}=\sum_{s=1}^{t}(\nabla_{i}f(\bx_{s}))^{2},\forall t\in\N$
(introduced in Lemma \ref{lem:AdaGrad-v}), we hence have
\begin{align}
\E\left[\sqrt{\bu_{T,i}}\right] & =\E\left[\sqrt{\frac{\bu_{T,i}}{\lambda+\sqrt{\bv_{T,i}+\bu_{T,i}+\bc_{i}^{2}}}\times\left(\lambda+\sqrt{\bv_{T,i}+\bu_{T,i}+\bc_{i}^{2}}\right)}\right]\nonumber \\
 & \overset{(c)}{\leq}\sqrt{\E\left[\frac{\bu_{T,i}}{\lambda+\sqrt{\bv_{T,i}+\bu_{T,i}+\bc_{i}^{2}}}\right]\E\left[\lambda+\sqrt{\bv_{T,i}+\bu_{T,i}+\bc_{i}^{2}}\right]}\nonumber \\
 & \leq\sqrt{\E\left[\sum_{t=1}^{T}\frac{(\nabla_{i}f(\bx_{t}))^{2}}{\lambda+\sqrt{\bw_{t,i}}}\right]\E\left[\lambda+\sqrt{\bv_{T,i}+\bu_{T,i}+\bc_{i}^{2}}\right]}\nonumber \\
 & \overset{(d)}{\leq}\sqrt{\E\left[\sum_{t=1}^{T}\frac{(\nabla_{i}f(\bx_{t}))^{2}}{\lambda+\sqrt{\bw_{t,i}}}\right]\left(\lambda+\bc_{i}+\sqrt{2}\bsig_{i}T^{\frac{1}{\p}}+\sqrt{3}\E\left[\sqrt{\bu_{T,i}}\right]\right)}\nonumber \\
 & \overset{(\ref{eq:AdaGrad-ub-c})}{=}\sqrt{\E\left[\sum_{t=1}^{T}\frac{(\nabla_{i}f(\bx_{t}))^{2}}{\lambda+\sqrt{\bw_{t,i}}}\right]\left(\lambda+\frac{\bsig_{i}}{D_{T,i}^{\frac{1}{2}-\frac{1}{\cp}}}T^{\frac{1}{2}-\frac{1}{\cp}}+\sqrt{2}\bsig_{i}T^{\frac{1}{\p}}+\sqrt{3}\E\left[\sqrt{\bu_{T,i}}\right]\right)},\label{eq:AdaGrad-ub-2}
\end{align}
where $(c)$ is by H\"{o}lder's inequality and $(d)$ follows similar
steps to proving Lemma \ref{lem:AdaGrad-v}. Now, we consider two
cases:
\begin{casenv}
\item $D_{T,i}\geq1$: in this case, we have
\[
\frac{\bsig_{i}}{D_{T,i}^{\frac{1}{2}-\frac{1}{\cp}}}T^{\frac{1}{2}-\frac{1}{\cp}}\leq\bsig_{i}T^{\frac{1}{2}-\frac{1}{\cp}}=\bsig_{i}T^{\frac{1}{\p}-\frac{1}{2}}\leq\bsig_{i}T^{\frac{1}{\p}},
\]
which implies that
\[
\E\left[\sqrt{\bu_{T,i}}\right]\overset{(\ref{eq:AdaGrad-ub-2})}{\leq}\sqrt{\E\left[\sum_{t=1}^{T}\frac{(\nabla_{i}f(\bx_{t}))^{2}}{\lambda+\sqrt{\bw_{t,i}}}\right]\left(\lambda+\left(1+\sqrt{2}\right)\bsig_{i}T^{\frac{1}{\p}}+\sqrt{3}\E\left[\sqrt{\bu_{T,i}}\right]\right)}.
\]
\item $D_{T,i}<1$: in this case, we have
\[
1>D_{T,i}\overset{(\ref{eq:AdaGrad-ub-D})}{=}2\ln\left(1+\frac{\sqrt{2}\bsig_{i}T^{\frac{1}{\p}}+\E\left[\sqrt{2\bu_{T,i}}\right]}{\lambda}\right),
\]
which implies that
\[
\E\left[\sqrt{\bu_{T,i}}\right]\leq\bsig_{i}T^{\frac{1}{\p}}+\E\left[\sqrt{\bu_{T,i}}\right]\leq\frac{\sqrt{e}-1}{\sqrt{2}}\lambda.
\]
\end{casenv}
Therefore, we always have
\[
\E\left[\sqrt{\bu_{T,i}}\right]\leq\sqrt{\E\left[\sum_{t=1}^{T}\frac{(\nabla_{i}f(\bx_{t}))^{2}}{\lambda+\sqrt{\bw_{t,i}}}\right]\left(\lambda+\left(1+\sqrt{2}\right)\bsig_{i}T^{\frac{1}{\p}}+\sqrt{3}\E\left[\sqrt{\bu_{T,i}}\right]\right)}+\frac{\sqrt{e}-1}{\sqrt{2}}\lambda.
\]
Sum up the above inequality for all $i\in\left[d\right]$ to have
\begin{align}
\E\left[\sum_{i=1}^{d}\sqrt{\bu_{T,i}}\right] & \leq\sum_{i=1}^{d}\sqrt{\E\left[\sum_{t=1}^{T}\frac{(\nabla_{i}f(\bx_{t}))^{2}}{\lambda+\sqrt{\bw_{t,i}}}\right]\left(\lambda+\left(1+\sqrt{2}\right)\bsig_{i}T^{\frac{1}{\p}}+\sqrt{3}\E\left[\sqrt{\bu_{T,i}}\right]\right)}+\frac{\sqrt{e}-1}{\sqrt{2}}d\lambda\nonumber \\
 & \overset{(e)}{\leq}\sqrt{\left(\sum_{i=1}^{d}\E\left[\sum_{t=1}^{T}\frac{(\nabla_{i}f(\bx_{t}))^{2}}{\lambda+\sqrt{\bw_{t,i}}}\right]\right)\left(d\lambda+\left(1+\sqrt{2}\right)\left\Vert \bsig\right\Vert _{1}T^{\frac{1}{\p}}+\sqrt{3}\E\left[\sum_{i=1}^{d}\sqrt{\bu_{T,i}}\right]\right)}+\frac{\sqrt{e}-1}{\sqrt{2}}d\lambda\nonumber \\
 & =\sqrt{\E\left[\sum_{t=1}^{T}\left\Vert \frac{(\nabla f(\bx_{t}))^{2}}{\lambda+\sqrt{\bw_{t}}}\right\Vert _{1}\right]\left(d\lambda+\left(1+\sqrt{2}\right)\left\Vert \bsig\right\Vert _{1}T^{\frac{1}{\p}}+\sqrt{3}\E\left[\sum_{i=1}^{d}\sqrt{\bu_{T,i}}\right]\right)}+\frac{\sqrt{e}-1}{\sqrt{2}}d\lambda\nonumber \\
\Rightarrow\E\left[\sum_{i=1}^{d}\sqrt{\bu_{T,i}}\right] & \leq\O\left(d\lambda+\E\left[\sum_{t=1}^{T}\left\Vert \frac{(\nabla f(\bx_{t}))^{2}}{\lambda+\sqrt{\bw_{t}}}\right\Vert _{1}\right]+\sqrt{\E\left[\sum_{t=1}^{T}\left\Vert \frac{(\nabla f(\bx_{t}))^{2}}{\lambda+\sqrt{\bw_{t}}}\right\Vert _{1}\right]\left\Vert \bsig\right\Vert _{1}T^{\frac{1}{\p}}}\right),\label{eq:AdaGrad-ub-3}
\end{align}
where $(e)$ is due to Cauchy-Schwarz inequality.

Now, we combine (\ref{eq:AdaGrad-ub-1}) and (\ref{eq:AdaGrad-ub-3})
to obtain
\begin{align*}
\E\left[\sum_{i=1}^{d}\sqrt{\bu_{T,i}}\right]\leq & \O\left(d\lambda+\frac{\Delta}{\gamma}+\gamma\left\Vert \bL\right\Vert _{1}\ln\res_{T}+\left\Vert \bsig\right\Vert _{1}T^{\frac{1}{\p}-\frac{1}{2}}\ln^{\frac{1}{\cp}+\frac{1}{2}}\res_{T}\right.\\
 & \left.\quad+\sqrt{\left(\frac{\Delta}{\gamma}+\gamma\left\Vert \bL\right\Vert _{1}\ln\res_{T}+\left\Vert \bsig\right\Vert _{1}T^{\frac{1}{\p}-\frac{1}{2}}\ln^{\frac{1}{\cp}+\frac{1}{2}}\res_{T}\right)\left\Vert \bsig\right\Vert _{1}T^{\frac{1}{\p}}}\right).
\end{align*}
Lastly, we observe that
\[
\sum_{i=1}^{d}\sqrt{\bu_{T,i}}=\sum_{i=1}^{d}\sqrt{\sum_{t=1}^{T}(\nabla_{i}f(\bx_{t}))^{2}}\geq\frac{\sum_{i=1}^{d}\sum_{t=1}^{T}\left|\nabla_{i}f(\bx_{t})\right|}{\sqrt{T}}=\frac{\sum_{t=1}^{T}\left\Vert \nabla f(\bx_{t})\right\Vert _{1}}{\sqrt{T}},
\]
which gives us
\begin{align*}
\E\left[\frac{1}{T}\sum_{t=1}^{T}\left\Vert \nabla f(\bx_{t})\right\Vert _{1}\right]\leq & \O\left(\frac{d\lambda+\frac{\Delta}{\gamma}+\gamma\left\Vert \bL\right\Vert _{1}\ln\res_{T}}{\sqrt{T}}+\frac{\left\Vert \bsig\right\Vert _{1}\ln^{\frac{1}{\cp}+\frac{1}{2}}\res_{T}}{T^{\frac{\p-1}{\p}}}\right.\\
 & \left.\quad+\frac{\sqrt{\left(\frac{\Delta}{\gamma}+\gamma\left\Vert \bL\right\Vert _{1}\ln\res_{T}\right)\left\Vert \bsig\right\Vert _{1}}}{T^{\frac{\p-1}{2\p}}}+\frac{\left\Vert \bsig\right\Vert _{1}\ln^{\frac{1}{2\cp}+\frac{1}{4}}\res_{T}}{T^{\frac{3\p-4}{4\p}}}\right).
\end{align*}
\end{proof}

\subsection{Helpful Lemmas}

To prove Theorem \ref{thm:AdaGrad-ub}, we require the following four
lemmas. Before presenting them, we recall a key ingredient in our
analysis, the generalized proxy stepsize, defined as follows
\[
\bw_{t}\defeq\bv_{t-1}+(\nabla f(\bx_{t}))^{2}+\bc^{2}\in\F_{t-1},
\]
where $\bc\in\R_{\geq0}^{d}$ is a free parameter that will be specified
in the final proof. As discussed in Section \ref{sec:AdaGrad}, it
plays a crucial role in establishing the final convergence rate.

We are now ready to give the first result, Lemma \ref{lem:AdaGrad-core},
which characterizes the per-iteration progress of $\AdaGrad$.
\begin{lem}
\label{lem:AdaGrad-core}Under Assumptions \ref{assu:smooth}, \ref{assu:unbias},
and \ref{assu:heavy}, for any $\bc\in\R_{\geq0}^{d}$ and $t\in\N$,
$\AdaGrad$ (Algorithm \ref{alg:AdaGrad}) guarantees
\[
\frac{\gamma}{2}\E\left[\left\Vert \frac{(\nabla f(\bx_{t}))^{2}}{\lambda+\sqrt{\bw_{t}}}\right\Vert _{1}\right]\leq\E\left[f(\bx_{t})-f(\bx_{t+1})\right]+\gamma\sum_{i=1}^{d}\frac{\bsig_{i}^{2}}{\bc_{i}}\left(\E\left[\frac{\bg_{t,i}^{2}}{\lambda^{2}+\bv_{t,i}}\right]\right)^{\frac{2}{\cp}}+\gamma\sum_{i=1}^{d}\left(\bc_{i}+\frac{\gamma\bL_{i}}{2}\right)\E\left[\frac{\bg_{t,i}^{2}}{\lambda^{2}+\bv_{t,i}}\right].
\]
\end{lem}

\begin{proof}
We start with Assumption \ref{assu:smooth} and use the update rule
of $\AdaGrad$ to obtain
\begin{align*}
f(\bx_{t+1}) & \leq f(\bx_{t})+\left\langle \nabla f(\bx_{t}),\bx_{t+1}-\bx_{t}\right\rangle +\frac{\left\Vert \bx_{t+1}-\bx_{t}\right\Vert _{\bL}^{2}}{2}\\
 & =f(\bx_{t})-\gamma\left\langle \nabla f(\bx_{t}),\frac{\bg_{t}}{\lambda+\sqrt{\bv_{t}}}\right\rangle +\frac{\gamma^{2}}{2}\left\Vert \frac{\bg_{t}}{\lambda+\sqrt{\bv_{t}}}\right\Vert _{\bL}^{2}\\
 & \leq f(\bx_{t})-\gamma\left\langle \nabla f(\bx_{t}),\frac{\bg_{t}}{\lambda+\sqrt{\bv_{t}}}\right\rangle +\frac{\gamma^{2}}{2}\left\Vert \frac{\bL\bg_{t}^{2}}{\lambda^{2}+\bv_{t}}\right\Vert _{1}\\
 & =f(\bx_{t})-\gamma\left\langle \nabla f(\bx_{t}),\frac{\bg_{t}}{\lambda+\sqrt{\bw_{t}}}\right\rangle +\gamma\left\langle \nabla f(\bx_{t}),\frac{\bg_{t}}{\lambda+\sqrt{\bw_{t}}}-\frac{\bg_{t}}{\lambda+\sqrt{\bv_{t}}}\right\rangle +\frac{\gamma^{2}}{2}\left\Vert \frac{\bL\bg_{t}^{2}}{\lambda^{2}+\bv_{t}}\right\Vert _{1}.
\end{align*}
Take conditional expectations on both sides of the above inequality
to obtain
\begin{align}
 & \E_{t-1}\left[f(\bx_{t+1})\right]\nonumber \\
\leq & f(\bx_{t})-\gamma\left\Vert \frac{(\nabla f(\bx_{t}))^{2}}{\lambda+\sqrt{\bw_{t}}}\right\Vert _{1}+\gamma\E_{t-1}\left[\left\langle \nabla f(\bx_{t}),\frac{\bg_{t}}{\lambda+\sqrt{\bw_{t}}}-\frac{\bg_{t}}{\lambda+\sqrt{\bv_{t}}}\right\rangle \right]+\frac{\gamma^{2}}{2}\E_{t-1}\left[\left\Vert \frac{\bL\bg_{t}^{2}}{\lambda^{2}+\bv_{t}}\right\Vert _{1}\right]\nonumber \\
= & f(\bx_{t})-\gamma\left\Vert \frac{(\nabla f(\bx_{t}))^{2}}{\lambda+\sqrt{\bw_{t}}}\right\Vert _{1}+\gamma\sum_{i=1}^{d}\E_{t-1}\left[\frac{\nabla_{i}f(\bx_{t})\bg_{t,i}}{\lambda+\sqrt{\bw_{t,i}}}-\frac{\nabla_{i}f(\bx_{t})\bg_{t,i}}{\lambda+\sqrt{\bv_{t,i}}}\right]+\frac{\gamma^{2}}{2}\sum_{i=1}^{d}\E_{t-1}\left[\frac{\bL_{i}\bg_{t,i}^{2}}{\lambda^{2}+\bv_{t,i}}\right],\label{eq:AdaGrad-core-1}
\end{align}
where the first step is by 
\[
\E_{t-1}\left[\left\langle \nabla f(\bx_{t}),\frac{\bg_{t}}{\lambda+\sqrt{\bw_{t}}}\right\rangle \right]\overset{\bx_{t},\bw_{t}\in\F_{t-1}}{=}\left\langle \nabla f(\bx_{t}),\frac{\E_{t-1}\left[\bg_{t}\right]}{\lambda+\sqrt{\bw_{t}}}\right\rangle \overset{\text{Assumption }\ref{assu:unbias}}{=}\left\langle \nabla f(\bx_{t}),\frac{\nabla f(\bx_{t})}{\lambda+\sqrt{\bw_{t}}}\right\rangle =\left\Vert \frac{(\nabla f(\bx_{t}))^{2}}{\lambda+\sqrt{\bw_{t}}}\right\Vert _{1}.
\]

For any fixed coordinate $i\in\left[d\right]$, we can bound
\begin{align}
 & \E_{t-1}\left[\frac{\nabla_{i}f(\bx_{t})\bg_{t,i}}{\lambda+\sqrt{\bw_{t,i}}}-\frac{\nabla_{i}f(\bx_{t})\bg_{t,i}}{\lambda+\sqrt{\bv_{t,i}}}\right]=\E_{t-1}\left[\frac{\left(\bg_{t,i}^{2}-(\nabla_{i}f(\bx_{t}))^{2}-\bc_{i}^{2}\right)\nabla_{i}f(\bx_{t})\bg_{t,i}}{(\lambda+\sqrt{\bw_{t,i}})(\lambda+\sqrt{\bv_{t,i}})(\sqrt{\bw_{t,i}}+\sqrt{\bv_{t,i}})}\right]\nonumber \\
= & \E_{t-1}\left[\frac{(\bg_{t,i}+\nabla_{i}f(\bx_{t}))\bxi_{t,i}\nabla_{i}f(\bx_{t})\bg_{t,i}}{(\lambda+\sqrt{\bw_{t,i}})(\lambda+\sqrt{\bv_{t,i}})(\sqrt{\bw_{t,i}}+\sqrt{\bv_{t,i}})}\right]+\E_{t-1}\left[\frac{-\bc_{i}^{2}\nabla_{i}f(\bx_{t})\bg_{t,i}}{(\lambda+\sqrt{\bw_{t,i}})(\lambda+\sqrt{\bv_{t,i}})(\sqrt{\bw_{t,i}}+\sqrt{\bv_{t,i}})}\right]\nonumber \\
\leq & \E_{t-1}\left[\frac{\left(\left|\bg_{t,i}\right|+\left|\nabla_{i}f(\bx_{t})\right|\right)\left|\bxi_{t,i}\right|\left|\nabla_{i}f(\bx_{t})\right|\left|\bg_{t,i}\right|}{(\lambda+\sqrt{\bw_{t,i}})(\lambda+\sqrt{\bv_{t,i}})(\sqrt{\bw_{t,i}}+\sqrt{\bv_{t,i}})}\right]+\E_{t-1}\left[\frac{\bc_{i}^{2}\left|\nabla_{i}f(\bx_{t})\right|\left|\bg_{t,i}\right|}{(\lambda+\sqrt{\bw_{t,i}})(\lambda+\sqrt{\bv_{t,i}})(\sqrt{\bw_{t,i}}+\sqrt{\bv_{t,i}})}\right]\nonumber \\
\overset{(a)}{\leq} & \E_{t-1}\left[\frac{\left|\bxi_{t,i}\right|\left|\nabla_{i}f(\bx_{t})\right|\left|\bg_{t,i}\right|}{(\lambda+\sqrt{\bw_{t,i}})(\lambda+\sqrt{\bv_{t,i}})}\right]+\E_{t-1}\left[\frac{\bc_{i}\left|\nabla_{i}f(\bx_{t})\right|\left|\bg_{t,i}\right|}{(\lambda+\sqrt{\bw_{t,i}})(\lambda+\sqrt{\bv_{t,i}})}\right]\nonumber \\
\overset{(b)}{\leq} & \frac{\left|\nabla_{i}f(\bx_{t})\right|}{\lambda+\sqrt{\bw_{t,i}}}\left(\E_{t-1}\left[\frac{\left|\bxi_{t,i}\right|\left|\bg_{t,i}\right|}{\lambda+\sqrt{\bv_{t,i}}}\right]+\E_{t-1}\left[\frac{\bc_{i}\left|\bg_{t,i}\right|}{\lambda+\sqrt{\bv_{t,i}}}\right]\right)\nonumber \\
\overset{(c)}{\leq} & \frac{(\nabla_{i}f(\bx_{t}))^{2}}{2(\lambda+\sqrt{\bw_{t,i}})}+\frac{1}{\lambda+\sqrt{\bw_{t,i}}}\left(\left(\E_{t-1}\left[\frac{\left|\bxi_{t,i}\right|\left|\bg_{t,i}\right|}{\lambda+\sqrt{\bv_{t,i}}}\right]\right)^{2}+\left(\E_{t-1}\left[\frac{\bc_{i}\left|\bg_{t,i}\right|}{\lambda+\sqrt{\bv_{t,i}}}\right]\right)^{2}\right),\label{eq:AdaGrad-core-2}
\end{align}
where $(a)$ is due to $\left|\bg_{t,i}\right|+\left|\nabla_{i}f(\bx_{t})\right|\leq\sqrt{\bw_{t,i}}+\sqrt{\bv_{t,i}}$
and $\bc_{i}\leq\sqrt{\bw_{t,i}}+\sqrt{\bv_{t,i}}$, $(b)$ is from
$\frac{\left|\nabla_{i}f(\bx_{t})\right|}{\lambda+\sqrt{\bw_{t,i}}}\in\F_{t-1}$,
and $(c)$ holds by AM-GM inequality, i.e., $\left|\nabla_{i}f(\bx_{t})\right|X\leq\frac{(\nabla_{i}f(\bx_{t}))^{2}}{4}+X^{2}$
for $X=\E_{t-1}\left[\frac{\left|\bxi_{t,i}\right|\left|\bg_{t,i}\right|}{\lambda+\sqrt{\bv_{t,i}}}\right]$
and $\E_{t-1}\left[\frac{\bc_{i}\left|\bg_{t,i}\right|}{\lambda+\sqrt{\bv_{t,i}}}\right]$,
respectively. Next, we apply H\"{o}lder's inequality to get
\begin{align}
 & \left(\E_{t-1}\left[\frac{\left|\bxi_{t,i}\right|\left|\bg_{t,i}\right|}{\lambda+\sqrt{\bv_{t,i}}}\right]\right)^{2}\leq\left(\E_{t-1}\left[\left|\bxi_{t,i}\right|^{\p}\right]\right)^{\frac{2}{\p}}\left(\E_{t-1}\left[\frac{\left|\bg_{t,i}\right|^{\cp}}{(\lambda+\sqrt{\bv_{t,i}})^{\cp}}\right]\right)^{\frac{2}{\cp}}\nonumber \\
\overset{\text{Assumption }\ref{assu:heavy}}{\leq} & \bsig_{i}^{2}\left(\E_{t-1}\left[\frac{\left|\bg_{t,i}\right|^{\cp}}{(\lambda+\sqrt{\bv_{t,i}})^{\cp}}\right]\right)^{\frac{2}{\cp}}\leq\bsig_{i}^{2}\left(\E_{t-1}\left[\left(\frac{\bg_{t,i}^{2}}{\lambda^{2}+\bv_{t,i}}\right)^{\frac{\cp}{2}}\right]\right)^{\frac{2}{\cp}},\label{eq:AdaGrad-core-3}
\end{align}
and
\begin{equation}
\left(\E_{t-1}\left[\frac{\bc_{i}\left|\bg_{t,i}\right|}{\lambda+\sqrt{\bv_{t,i}}}\right]\right)^{2}\leq\E_{t-1}\left[\frac{\bc_{i}^{2}\bg_{t,i}^{2}}{(\lambda+\sqrt{\bv_{t,i}})^{2}}\right]\leq\E_{t-1}\left[\frac{\bc_{i}^{2}\bg_{t,i}^{2}}{\lambda^{2}+\bv_{t,i}}\right].\label{eq:AdaGrad-core-4}
\end{equation}
Plug (\ref{eq:AdaGrad-core-3}) and (\ref{eq:AdaGrad-core-4}) into
(\ref{eq:AdaGrad-core-2}) to obtain
\begin{align}
 & \E_{t-1}\left[\frac{\nabla_{i}f(\bx_{t})\bg_{t,i}}{\lambda+\sqrt{\bw_{t,i}}}-\frac{\nabla_{i}f(\bx_{t})\bg_{t,i}}{\lambda+\sqrt{\bv_{t,i}}}\right]\nonumber \\
\leq & \frac{(\nabla_{i}f(\bx_{t}))^{2}}{2(\lambda+\sqrt{\bw_{t,i}})}+\frac{\bsig_{i}^{2}}{\lambda+\sqrt{\bw_{t,i}}}\left(\E_{t-1}\left[\left(\frac{\bg_{t,i}^{2}}{\lambda^{2}+\bv_{t,i}}\right)^{\frac{\cp}{2}}\right]\right)^{\frac{2}{\cp}}+\frac{\bc_{i}^{2}}{\lambda+\sqrt{\bw_{t,i}}}\E_{t-1}\left[\frac{\bg_{t,i}^{2}}{\lambda^{2}+\bv_{t,i}}\right]\nonumber \\
\leq & \frac{(\nabla_{i}f(\bx_{t}))^{2}}{2(\lambda+\sqrt{\bw_{t,i}})}+\frac{\bsig_{i}^{2}}{\bc_{i}}\left(\E_{t-1}\left[\frac{\bg_{t,i}^{2}}{\lambda^{2}+\bv_{t,i}}\right]\right)^{\frac{2}{\cp}}+\bc_{i}\E_{t-1}\left[\frac{\bg_{t,i}^{2}}{\lambda^{2}+\bv_{t,i}}\right],\label{eq:AdaGrad-core-5}
\end{align}
where the last step is by $\lambda+\sqrt{\bw_{t,i}}\geq\bc_{i}$,
$\frac{\bg_{t,i}^{2}}{\lambda^{2}+\bv_{t,i}}\leq1$, and $\frac{\cp}{2}\geq1$.

We combine (\ref{eq:AdaGrad-core-1}) and (\ref{eq:AdaGrad-core-5})
to have
\begin{align*}
\E_{t-1}\left[f(\bx_{t+1})\right]\leq & f(\bx_{t})-\frac{\gamma}{2}\left\Vert \frac{(\nabla f(\bx_{t}))^{2}}{\lambda+\sqrt{\bw_{t}}}\right\Vert _{1}+\gamma\sum_{i=1}^{d}\frac{\bsig_{i}^{2}}{\bc_{i}}\left(\E_{t-1}\left[\frac{\bg_{t,i}^{2}}{\lambda^{2}+\bv_{t,i}}\right]\right)^{\frac{2}{\cp}}\\
 & +\gamma\sum_{i=1}^{d}\left(\bc_{i}+\frac{\gamma\bL_{i}}{2}\right)\E_{t-1}\left[\frac{\bg_{t,i}^{2}}{\lambda^{2}+\bv_{t,i}}\right].
\end{align*}
Taking expectations on both sides and rearranging terms, we know
\begin{align*}
\frac{\gamma}{2}\E\left[\left\Vert \frac{(\nabla f(\bx_{t}))^{2}}{\lambda+\sqrt{\bw_{t}}}\right\Vert _{1}\right]\leq & \E\left[f(\bx_{t})-f(\bx_{t+1})\right]+\gamma\sum_{i=1}^{d}\frac{\bsig_{i}^{2}}{\bc_{i}}\E\left[\left(\E_{t-1}\left[\frac{\bg_{t,i}^{2}}{\lambda^{2}+\bv_{t,i}}\right]\right)^{\frac{2}{\cp}}\right]\\
 & +\gamma\sum_{i=1}^{d}\left(\bc_{i}+\frac{\gamma\bL_{i}}{2}\right)\E\left[\frac{\bg_{t,i}^{2}}{\lambda^{2}+\bv_{t,i}}\right].
\end{align*}
Finally, noticing that $\frac{\cp}{2}\geq1$, we hence can invoke
H\"{o}lder's inequality again to have, for any $i\in\left[d\right]$,
\[
\E\left[\left(\E_{t-1}\left[\frac{\bg_{t,i}^{2}}{\lambda^{2}+\bv_{t,i}}\right]\right)^{\frac{2}{\cp}}\right]\leq\left(\E\left[\E_{t-1}\left[\frac{\bg_{t,i}^{2}}{\lambda^{2}+\bv_{t,i}}\right]\right]\right)^{\frac{2}{\cp}}=\left(\E\left[\frac{\bg_{t,i}^{2}}{\lambda^{2}+\bv_{t,i}}\right]\right)^{\frac{2}{\cp}},
\]
which leads us to the desired result.
\end{proof}

Lemma \ref{lem:AdaGrad-general-ln} can be viewed as a generalization
of the existing inequality in the literature (see, e.g., \citet{pmlr-v97-ward19a})
from $q=1$ to any $q\in\left[0,1\right]$.
\begin{lem}
\label{lem:AdaGrad-general-ln}For any $T\in\N$, $i\in\left[d\right]$,
and $q\in\left[0,1\right]$, $\AdaGrad$ (Algorithm \ref{alg:AdaGrad})
guarantees
\[
\sum_{t=1}^{T}\left(\E\left[\frac{\bg_{t,i}^{2}}{\lambda^{2}+\bv_{t,i}}\right]\right)^{q}\leq T^{1-q}\left(\E\left[\ln\left(1+\frac{\bv_{T,i}}{\lambda^{2}}\right)\right]\right)^{q}.
\]
\end{lem}

\begin{proof}
By the concavity of $x^{q}$ (since $q\in\left[0,1\right]$), we have
\[
\frac{1}{T}\sum_{t=1}^{T}\left(\E\left[\frac{\bg_{t,i}^{2}}{\lambda^{2}+\bv_{t,i}}\right]\right)^{q}\leq\left(\frac{1}{T}\sum_{t=1}^{T}\E\left[\frac{\bg_{t,i}^{2}}{\lambda^{2}+\bv_{t,i}}\right]\right)^{q},
\]
which implies that
\begin{align*}
\sum_{t=1}^{T}\left(\E\left[\frac{\bg_{t,i}^{2}}{\lambda^{2}+\bv_{t,i}}\right]\right)^{q} & \leq T^{1-q}\left(\E\left[\sum_{t=1}^{T}\frac{\bg_{t,i}^{2}}{\lambda^{2}+\bv_{t,i}}\right]\right)^{q}=T^{1-q}\left(\E\left[\sum_{t=1}^{T}1-\frac{\lambda^{2}+\bv_{t-1,i}}{\lambda^{2}+\bv_{t,i}}\right]\right)^{q}\\
 & \overset{(a)}{\leq}T^{1-q}\left(\E\left[\sum_{t=1}^{T}\ln\left(\frac{\lambda^{2}+\bv_{t,i}}{\lambda^{2}+\bv_{t-1,i}}\right)\right]\right)^{q}=T^{1-q}\left(\E\left[\ln\left(1+\frac{\bv_{T,i}}{\lambda^{2}}\right)\right]\right)^{q},
\end{align*}
where $(a)$ is due to $1-x^{-1}\leq\ln x,\forall x>0$.
\end{proof}

The next Lemma \ref{lem:AdaGrad-v} is the coordinate-wise version
of Lemma \ref{lem:main-AdaGradNorm-v}.
\begin{lem}
\label{lem:AdaGrad-v}Under Assumption \ref{assu:heavy}, for any
$t\in\N$ and $i\in\left[d\right]$, $\AdaGrad$ (Algorithm \ref{alg:AdaGrad})
guarantees
\[
\E\left[\sqrt{\bv_{t,i}}\right]\leq\sqrt{2}\bsig_{i}t^{\frac{1}{\p}}+\E\left[\sqrt{2\bu_{t,i}}\right],
\]
where $\bu_{t,i}\triangleq\sum_{s=1}^{t}(\nabla_{i}f(\bx_{s}))^{2},\forall t\in\N$.
\end{lem}

\begin{proof}
By the definition of $\bv_{t,i}$, we have
\begin{align*}
\sqrt{\bv_{t,i}} & =\sqrt{\sum_{s=1}^{t}\bg_{s,i}^{2}}=\sqrt{\sum_{s=1}^{t}(\bxi_{s,i}+\nabla_{i}f(\bx_{s}))^{2}}\leq\sqrt{2\sum_{s=1}^{t}\bxi_{s,i}^{2}+2\sum_{s=1}^{t}(\nabla_{i}f(\bx_{s}))^{2}}\\
 & \leq\sqrt{2\sum_{s=1}^{t}\bxi_{s,i}^{2}}+\sqrt{2\bu_{t,i}}\leq\sqrt{2}\left(\sum_{s=1}^{t}\left|\bxi_{s,i}\right|^{\p}\right)^{\frac{1}{\p}}+\sqrt{2\bu_{t,i}},
\end{align*}
where the last step is due to $\left\Vert \cdot\right\Vert _{2}\leq\left\Vert \cdot\right\Vert _{\p}$
when $\p\in\left[1,2\right]$. By H\"{o}lder's inequality, we conclude
\[
\E\left[\sqrt{\bv_{t,i}}\right]\leq\sqrt{2}\left(\E\left[\sum_{s=1}^{t}\left|\bxi_{s,i}\right|^{\p}\right]\right)^{\frac{1}{\p}}+\E\left[\sqrt{2\bu_{t,i}}\right]\overset{\text{Assumption }\ref{assu:heavy}}{\leq}\sqrt{2}\bsig_{i}t^{\frac{1}{\p}}+\E\left[\sqrt{2\bu_{t,i}}\right].
\]
\end{proof}

Finally, we prove Lemma \ref{lem:AdaGrad-ln-bound}, which is also
inspired by \citet{pmlr-v97-ward19a}.
\begin{lem}
\label{lem:AdaGrad-ln-bound}Under Assumptions \ref{assu:smooth}
and \ref{assu:heavy}, for any $i\in\left[d\right]$, $\AdaGrad$
(Algorithm \ref{alg:AdaGrad}) guarantees
\[
\E\left[\ln\left(1+\frac{\bv_{T,i}}{\lambda^{2}}\right)\right]\leq2\ln\left(1+\frac{\sqrt{2}\bsig_{i}T^{\frac{1}{\p}}+\E\left[\sqrt{2\bu_{T,i}}\right]}{\lambda}\right)\leq2\ln\res_{T},
\]
where $\res_{T}\defeq1+\frac{\sqrt{2}\left\Vert \bsig\right\Vert _{\infty}T^{\frac{1}{\p}}+2\left\Vert \nabla f(\bx_{1})\right\Vert _{\infty}T^{\frac{1}{2}}+2\gamma\sqrt{\left\Vert \bL\right\Vert _{1}\left\Vert \bL\right\Vert _{\infty}}T^{\frac{3}{2}}}{\lambda}.$
\end{lem}

\begin{proof}
Note that
\[
\E\left[\ln\left(1+\frac{\bv_{T,i}}{\lambda^{2}}\right)\right]=2\E\left[\ln\left(\sqrt{1+\frac{\bv_{T,i}}{\lambda^{2}}}\right)\right]\leq2\E\left[\ln\left(1+\frac{\sqrt{\bv_{T,i}}}{\lambda}\right)\right]\leq2\ln\left(1+\frac{\E\left[\sqrt{\bv_{T,i}}\right]}{\lambda}\right),
\]
where the last step is due to the concavity of $\ln x$. Next, we
invoke Lemma \ref{lem:AdaGrad-v} to obtain
\begin{equation}
\E\left[\ln\left(1+\frac{\bv_{T,i}}{\lambda^{2}}\right)\right]\leq2\ln\left(1+\frac{\sqrt{2}\bsig_{i}T^{\frac{1}{\p}}+\E\left[\sqrt{2\bu_{T,i}}\right]}{\lambda}\right).\label{eq:AdaGrad-ln-bound-1}
\end{equation}

Moreover, under Assumption \ref{assu:smooth}, we have almost surely,
for any $t\in\left[T\right]$,
\begin{align*}
\left|\nabla_{i}f(\bx_{t})-\nabla_{i}f(\bx_{1})\right| & \leq\sqrt{\bL_{i}}\left\Vert \nabla f(\bx_{t})-\nabla f(\bx_{1})\right\Vert _{\nicefrac{1}{\bL}}\leq\sqrt{\bL_{i}}\left\Vert \bx_{t}-\bx_{1}\right\Vert _{\bL}\leq\sqrt{\bL_{i}}\sum_{s=1}^{t-1}\left\Vert \bx_{s+1}-\bx_{s}\right\Vert _{\bL}\\
 & =\sqrt{\bL_{i}}\sum_{s=1}^{t-1}\left\Vert \frac{\gamma}{\lambda+\sqrt{\bv_{s}}}\bg_{s}\right\Vert _{\bL}=\gamma\sqrt{\bL_{i}}\sum_{s=1}^{t-1}\sqrt{\sum_{j=1}^{d}\frac{\bL_{j}\bg_{s,j}^{2}}{(\lambda+\sqrt{\bv_{s,j}})^{2}}}\\
 & \leq\gamma\sqrt{\bL_{i}}\sum_{s=1}^{t-1}\sqrt{\sum_{j=1}^{d}\bL_{j}}=\gamma\sqrt{\bL_{i}\left\Vert \bL\right\Vert _{1}}(t-1)\\
\Rightarrow\left|\nabla_{i}f(\bx_{t})\right| & \leq\left|\nabla_{i}f(\bx_{1})\right|+\gamma\sqrt{\bL_{i}\left\Vert \bL\right\Vert _{1}}(t-1).
\end{align*}
Hence, there is almost surely
\begin{align}
\sqrt{2\bu_{T,i}} & =\sqrt{2\sum_{t=1}^{T}(\nabla_{i}f(\bx_{t}))^{2}}\leq\sqrt{2\sum_{t=1}^{T}\left(\left|\nabla_{i}f(\bx_{1})\right|+\gamma\sqrt{\bL_{i}\left\Vert \bL\right\Vert _{1}}(t-1)\right)^{2}}\nonumber \\
 & \leq2\sqrt{\sum_{t=1}^{T}\left|\nabla_{i}f(\bx_{1})\right|^{2}+\sum_{t=1}^{T}\gamma^{2}\bL_{i}\left\Vert \bL\right\Vert _{1}(t-1)^{2}}\nonumber \\
 & \leq2\left|\nabla_{i}f(\bx_{1})\right|T^{\frac{1}{2}}+2\gamma\sqrt{\bL_{i}\left\Vert \bL\right\Vert _{1}}T^{\frac{3}{2}}.\label{eq:AdaGrad-ln-bound-2}
\end{align}

Finally, we plug (\ref{eq:AdaGrad-ln-bound-2}) back into (\ref{eq:AdaGrad-ln-bound-1})
to have
\begin{align*}
 & 2\ln\left(1+\frac{\sqrt{2}\bsig_{i}T^{\frac{1}{\p}}+\E\left[\sqrt{2\bu_{T,i}}\right]}{\lambda}\right)\\
\leq & 2\ln\left(1+\frac{\sqrt{2}\bsig_{i}T^{\frac{1}{\p}}+2\left|\nabla_{i}f(\bx_{1})\right|T^{\frac{1}{2}}+2\gamma\sqrt{\bL_{i}\left\Vert \bL\right\Vert _{1}}T^{\frac{3}{2}}}{\lambda}\right)\\
\leq & 2\ln\left(1+\frac{\sqrt{2}\left\Vert \bsig\right\Vert _{\infty}T^{\frac{1}{\p}}+2\left\Vert \nabla f(\bx_{1})\right\Vert _{\infty}T^{\frac{1}{2}}+2\gamma\sqrt{\left\Vert \bL\right\Vert _{1}\left\Vert \bL\right\Vert _{\infty}}T^{\frac{3}{2}}}{\lambda}\right)=2\ln\res_{T}.
\end{align*}
\end{proof}

\section{Algorithm-Dependent Lower Bound for $\protect\AdaGrad$\label{sec:AdaGrad-lb}}

This section provides the full statement of Theorem \ref{thm:main-AdaGrad-lb-1d}
and its proof.

\subsection{Full Theorem and Its Proof}
\begin{thm}[Full statement of Theorem \ref{thm:main-AdaGrad-lb-1d}]
\label{thm:AdaGrad-lb-1d}Let $d=1$, for any given $\Delta>0$,
$L>0$, $\p\in\left(1,2\right]$, $\sigma\geq0$, $0<\epsilon\leq\sqrt{2\Delta L}$,
$\sx_{1}\in\R$, $\gamma>0$, $\lambda\geq0$ satisfying $\lambda=0$
when $\sigma=0$, there exists a function $f:\R\to\R$ associated
with a function $g:\R\times\left\{ 0,1\right\} \to\R$ and a Bernoulli
distribution $\P$ on $\left\{ 0,1\right\} $ satisfying
\begin{enumerate}
\item $f(\sx_{1})-\inf_{\sx\in\R}f(\sx)\leq\Delta$ and $f$ is $L$-smooth;
\item $\E_{r\sim\P}\left[g(\sx;r)\right]=f'(\sx)$ and $\E_{r\sim\P}\left[\left|g(\sx;r)-f'(\sx)\right|^{\p}\right]\leq\sigma^{\p}$.
\end{enumerate}
Moreover, if using $\AdaGrad$ (Algorithm \ref{alg:AdaGrad}), with
initial point $\sx_{1}$, learning rate $\gamma$, and hyperparameter
$\lambda$, to optimize $f$ by interacting with $g$ (i.e., $g_{t}=g(\sx_{t};r_{t})$
where $r_{t}\sim\P$ is independent of the history), one must use
at least
\[
\Omega\left(\frac{\lambda\Delta/\gamma+\Delta^{2}/\gamma^{2}+\gamma^{2}L^{2}\ln^{2}\left(\frac{\gamma L(1+\sigma^{\frac{\p}{\p-1}}\epsilon^{-\frac{\p}{\p-1}})}{\lambda+\epsilon(1+\sigma^{\frac{\p}{\p-1}}\epsilon^{-\frac{\p}{\p-1}})}\right)}{\epsilon^{2}}+\frac{\left(\Delta^{2}/\gamma^{2}+\gamma^{2}L^{2}\ln^{2}\left(\frac{\gamma L(1+\sigma^{\frac{\p}{\p-1}}\epsilon^{-\frac{\p}{\p-1}})}{\lambda+\epsilon(1+\sigma^{\frac{\p}{\p-1}}\epsilon^{-\frac{\p}{\p-1}})}\right)\right)\sigma^{\frac{\p}{\p-1}}}{\epsilon^{\frac{3\p-2}{\p-1}}}\right)
\]
iterations to make $\E\left[\frac{1}{T}\sum_{t=1}^{T}\left|f'(\sx_{t})\right|\right]<\frac{\epsilon}{2}$
for small enough $\epsilon$ (see (\ref{eq:AdaGrad-lb-1d-eps}) for
the precise condition on $\epsilon$).
\end{thm}

\begin{rem}
The technical condition $\lambda=0$ when $\sigma=0$ is imposed to
ensure that the second and third inequalities in (\ref{eq:AdaGrad-lb-1d-eps})
hold in the deterministic case. In fact, it suffices to require $\lambda\leq\O(\epsilon)$
when $\sigma=0$, but we set $\lambda=0$ for simplicity.
\end{rem}

\begin{proof}
In the proof, we write
\begin{equation}
q\defeq\frac{1}{\left[1+\frac{\p-1}{4}\left(\frac{2\sigma}{\epsilon}\right)^{\p}\right]^{\frac{1}{\p-1}}}\in\left[0,1\right].\label{eq:AdaGrad-lb-1d-q}
\end{equation}
Moreover, let
\begin{equation}
\delta_{t}\defeq\frac{\gamma}{\frac{\lambda q}{\epsilon}+\sqrt{t}},\label{eq:AdaGrad-lb-1d-delta}
\end{equation}
and
\begin{equation}
T_{\star}\defeq\inf\left\{ T\in\N:\Delta-\epsilon\sum_{t=1}^{T}\delta_{t}+\frac{L}{4}\sum_{t=1}^{T}\delta_{t}^{2}<\frac{\epsilon^{2}}{2L}\right\} .\label{eq:AdaGrad-lb-1d-Tstar}
\end{equation}
Now, let us consider the function $f$ constructed in Lemma \ref{lem:AdaGrad-lb-1d-construction}
and the following function $g:\R\times\left\{ 0,1\right\} \to\R$,
\begin{equation}
g(\sx;r)=\begin{cases}
f'(\sx) & \sx\notin\left\{ \sy_{1},\mydots,\sy_{T_{\star}}\right\} \\
\frac{r}{q}f'(x) & x\in\left\{ \sy_{1},\mydots,\sy_{T_{\star}}\right\} 
\end{cases},\label{eq:AdaGrad-lb-1d-g}
\end{equation}
where we recall that $\sy_{t}=\sx_{1}+\sum_{s=1}^{t-1}\delta_{s},\forall t\in\left[T_{\star}\right]$
is introduced in Lemma \ref{lem:AdaGrad-lb-1d-construction} satisfying
\begin{equation}
f'(y_{t})=-\epsilon,\forall t\in\left[T_{\star}\right].\label{eq:AdaGrad-lb-1d-f'(y)}
\end{equation}

According to Lemma \ref{lem:AdaGrad-lb-1d-construction}, we know
\begin{eqnarray*}
f(\sx_{1})-\inf_{\sx\in\R}f(\sx)\leq\Delta & \text{and} & f\text{ is }L\text{-smooth}.
\end{eqnarray*}
Next, let $\P$ be the Bernoulli distribution with the parameter $q$
given in (\ref{eq:AdaGrad-lb-1d-q}), i.e.,
\begin{eqnarray}
\P\left[r=0\right]=1-q & \text{and} & \P\left[r=1\right]=q.\label{eq:AdaGrad-lb-1d-prob}
\end{eqnarray}
One can find that $\E_{r\sim\P}\left[g(\sx;r)\right]\overset{(\ref{eq:AdaGrad-lb-1d-g}),(\ref{eq:AdaGrad-lb-1d-prob})}{=}f'(\sx)$
and $\E_{r\sim\P}\left[\left|g(\sx;r)-f'(\sx)\right|^{\p}\right]\overset{(\ref{eq:AdaGrad-lb-1d-g})}{=}0\leq\sigma^{\p}$
if $\sx\notin\left\{ \sy_{1},\mydots,\sy_{T_{\star}}\right\} $. If
$x\in\left\{ \sy_{1},\mydots,\sy_{T_{\star}}\right\} $, we know
\begin{align*}
\E_{r\sim\P}\left[\left|g(\sx;r)-f'(\sx)\right|^{\p}\right] & \overset{(\ref{eq:AdaGrad-lb-1d-g}),(\ref{eq:AdaGrad-lb-1d-prob})}{=}\left|f'(x)\right|^{\p}(1-q)+\left|\frac{f'(x)}{q}-f'(x)\right|^{\p}q\overset{(\ref{eq:AdaGrad-lb-1d-f'(y)})}{=}\left(1-q\right)\frac{q^{\p-1}+(1-q)^{\p-1}}{q^{\p-1}}\epsilon^{\p}\\
 & \overset{(a)}{\leq}\left(1-q\right)\frac{2^{2-\p}\epsilon^{\p}}{q^{\p-1}}\overset{(b)}{\leq}\frac{\left(1-q^{\p-1}\right)2^{2-\p}\epsilon^{\p}}{\left(\p-1\right)q^{\p-1}}\overset{(\ref{eq:AdaGrad-lb-1d-q})}{=}\sigma^{\p},
\end{align*}
where $(a)$ is by $\frac{q^{\p-1}+(1-q)^{\p-1}}{2}\leq\frac{1}{2^{\p-1}}$
due to the concavity of $\sx^{\p-1}$ (since $\p-1\in\left(0,1\right]$)
and $(b)$ holds by $1-q\leq\frac{1-q^{\p-1}}{\p-1},\forall q\in\left[0,1\right],\p\in\left(1,2\right]$.
Therefore, we know
\begin{eqnarray*}
\E_{r\sim\P}\left[g(\sx;r)\right]=f'(\sx) & \text{and} & \E_{r\sim\P}\left[\left|g(\sx;r)-f'(\sx)\right|^{\p}\right]\leq\sigma^{\p}.
\end{eqnarray*}

Suppose one runs $\AdaGrad$ to optimize $f$ by interacting with
$g$. Let us define
\begin{eqnarray}
R_{t}\defeq\sum_{s=1}^{t}r_{t},\forall t\in\N & \text{and} & E_{T}\defeq\left\{ R_{T}\leq T_{\star}-1\right\} ,\forall T\in\N.\label{eq:AdaGrad-lb-1d-R-E}
\end{eqnarray}
Given $T\in\N$, for any sample path in $E_{T}$, we use induction
to show
\begin{equation}
\left\{ \sx_{1},\mydots,\sx_{t}\right\} \subseteq\left\{ \sy_{1},\mydots,\sy_{T_{\star}}\right\} ,\forall t\in\left[T\right].\label{eq:AdaGrad-lb-1d-hypothesis}
\end{equation}
For $t=1$, (\ref{eq:AdaGrad-lb-1d-hypothesis}) is true since $\sy_{1}=\sx_{1}$.
Suppose (\ref{eq:AdaGrad-lb-1d-hypothesis}) holds for some $t\leq T-1$,
then we know 
\[
g_{s}=g(\sx_{s};r_{s})\overset{(\ref{eq:AdaGrad-lb-1d-g}),(\ref{eq:AdaGrad-lb-1d-f'(y)}),(\ref{eq:AdaGrad-lb-1d-hypothesis})}{=}-\frac{r_{s}}{q}\epsilon,\forall s\in\left[t\right],
\]
which implies that 
\[
v_{s}=\sum_{\ell=1}^{s}g_{\ell}^{2}=\frac{\epsilon^{2}}{q^{2}}\sum_{\ell=1}^{s}r_{\ell}^{2}=\frac{\epsilon^{2}}{q^{2}}\sum_{\ell=1}^{s}r_{\ell}\overset{(\ref{eq:AdaGrad-lb-1d-R-E})}{=}\frac{\epsilon^{2}}{q^{2}}R_{s},\forall s\in\left[t\right].
\]
Therefore, by the update rule of $\AdaGrad$,
\[
\sx_{t+1}=\sx_{1}-\sum_{s=1}^{t}\frac{\gamma}{\lambda+\sqrt{v_{s}}}g_{s}=\sx_{1}+\sum_{s=1}^{t}\frac{\gamma r_{s}}{\frac{\lambda q}{\epsilon}+\sqrt{R_{s}}}\overset{(\ref{eq:AdaGrad-lb-1d-delta})}{=}\sx_{1}+\sum_{s=1}^{t}\delta_{R_{s}}r_{s}=\sx_{1}+\sum_{s=1}^{R_{t}}\delta_{s}\in\left\{ \sy_{1},\mydots,\sy_{T_{\star}}\right\} ,
\]
where the last step is due to $R_{t}\leq R_{T}\leq T_{\star}-1$ and
$\sy_{t}=\sx_{1}+\sum_{s=1}^{t-1}\delta_{s},\forall t\in\left[T_{\star}\right]$.
Thus, the induction is complete.

If $T\leq\frac{T_{\star}-1}{2q}$, by Markov's inequality, we have
\[
\P\left[R_{T}>2qT\right]\leq\frac{\E\left[R_{T}\right]}{2qT}=\frac{1}{2}\Rightarrow\P\left[E_{T}\right]\geq\P\left[R_{T}\leq2qT\right]\geq\frac{1}{2},
\]
which implies that
\[
\E\left[\frac{1}{T}\sum_{t=1}^{T}\left|f'(\sx_{t})\right|\right]\geq\E\left[\frac{1}{T}\sum_{t=1}^{T}\left|f'(\sx_{t})\right|\mid E_{T}\right]\P\left[E_{T}\right]\geq\frac{\epsilon}{2},
\]
where the last step is due to $\left|f'(\sx_{t})\right|=\epsilon,\forall t\in\left[T\right]$
since $\left\{ \sx_{1},\mydots,\sx_{T}\right\} \subseteq\left\{ \sy_{1},\mydots,\sy_{T_{\star}}\right\} $
(see (\ref{eq:AdaGrad-lb-1d-hypothesis})) and $\left|f'(\sy_{t})\right|=\epsilon,\forall t\in\left[T_{\star}\right]$
(see (\ref{eq:AdaGrad-lb-1d-f'(y)})). Therefore, to make $\E\left[\frac{1}{T}\sum_{t=1}^{T}\left|f'(\sx_{t})\right|\right]<\frac{\epsilon}{2}$,
one must have
\begin{equation}
T>\frac{T_{\star}-1}{2q}.\label{eq:AdaGrad-lb-1d-T}
\end{equation}

Finally, let us assume $\epsilon$ is small enough to satisfy\footnote{The first condition is trivial to satisfy. To make the second one
hold, since $\ln(1+\frac{1}{x})\geq\frac{1}{x+1}$, it suffices to
ensure that $\frac{16\epsilon}{\gamma L}(\left(\frac{\lambda q}{\epsilon}\right)^{2}+2)\leq1$,
which is possible when $\epsilon$ is small enough due to the definition
of $q$ (see (\ref{eq:AdaGrad-lb-1d-q})). The third one is also true
when $\epsilon$ is small enough by the definition of $q$.}
\begin{eqnarray}
\epsilon\leq\sqrt{\Delta L}, & \ln\left(1+\frac{1}{\left(\frac{\lambda q}{\epsilon}\right)^{2}+1}\right)\geq\frac{16\epsilon}{\gamma L}, & \frac{16\epsilon}{\gamma L}\sqrt{\left(\frac{\lambda q}{\epsilon}\right)^{2}+1}\leq\frac{\ln c}{\sqrt{2c}},\label{eq:AdaGrad-lb-1d-eps}
\end{eqnarray}
where $c\in\left[3.92,3.93\right]$ is the unique positive solution
to $2c=(1+c)\ln(1+c)$. Note that
\begin{align*}
T_{\star} & \overset{(\ref{eq:AdaGrad-lb-1d-delta}),(\ref{eq:AdaGrad-lb-1d-Tstar})}{=}\inf\left\{ T\in\N:\Delta-\epsilon\gamma\sum_{t=1}^{T}\frac{1}{\frac{\lambda q}{\epsilon}+\sqrt{t}}+\frac{\gamma^{2}L}{4}\sum_{t=1}^{T}\frac{1}{\left(\frac{\lambda q}{\epsilon}+\sqrt{t}\right)^{2}}<\frac{\epsilon^{2}}{2L}\right\} \\
 & \overset{(\ref{eq:AdaGrad-lb-1d-eps})}{\geq}\inf\left\{ T\in\N:\frac{\Delta}{2}-\epsilon\gamma\sum_{t=1}^{T}\frac{1}{\frac{\lambda q}{\epsilon}+\sqrt{t}}+\frac{\gamma^{2}L}{4}\sum_{t=1}^{T}\frac{1}{\left(\frac{\lambda q}{\epsilon}+\sqrt{t}\right)^{2}}<0\right\} \\
 & =\inf\left\{ T\in\N:\frac{\Delta}{2\gamma\epsilon}+\frac{\gamma L}{4\epsilon}\sum_{t=1}^{T}\frac{1}{\left(\frac{\lambda q}{\epsilon}+\sqrt{t}\right)^{2}}<\sum_{t=1}^{T}\frac{1}{\frac{\lambda q}{\epsilon}+\sqrt{t}}\right\} .
\end{align*}
Moreover, we observe that
\begin{align*}
\sum_{t=1}^{T}\frac{1}{\frac{\lambda q}{\epsilon}+\sqrt{t}} & \leq\int_{0}^{T}\frac{\d t}{\frac{\lambda q}{\epsilon}+\sqrt{t}}\leq\int_{0}^{T}\frac{\d t}{\sqrt{\left(\frac{\lambda q}{\epsilon}\right)^{2}+t}}=2\sqrt{\left(\frac{\lambda q}{\epsilon}\right)^{2}+T}-\frac{2\lambda q}{\epsilon},\\
\sum_{t=1}^{T}\frac{1}{\left(\frac{\lambda q}{\epsilon}+\sqrt{t}\right)^{2}} & \geq\int_{1}^{T+1}\frac{\d t}{\left(\frac{\lambda q}{\epsilon}+\sqrt{t}\right)^{2}}\geq\int_{1}^{T+1}\frac{\d t}{2\left(\left(\frac{\lambda q}{\epsilon}\right)^{2}+t\right)}=\frac{\ln\left(1+\frac{T}{\left(\frac{\lambda q}{\epsilon}\right)^{2}+1}\right)}{2},
\end{align*}
which together imply that
\begin{align}
T_{\star} & \geq\inf\left\{ T\in\N:\frac{\Delta}{2\gamma\epsilon}+\frac{\gamma L}{8\epsilon}\ln\left(1+\frac{T}{\left(\frac{\lambda q}{\epsilon}\right)^{2}+1}\right)<2\sqrt{\left(\frac{\lambda q}{\epsilon}\right)^{2}+T}-\frac{2\lambda q}{\epsilon}\right\} \nonumber \\
 & \geq\frac{\inf\left\{ T\in\N:\frac{\Delta}{4\gamma\epsilon}<\sqrt{\left(\frac{\lambda q}{\epsilon}\right)^{2}+T}-\frac{\lambda q}{\epsilon}\right\} +\inf\left\{ T\in\N:\frac{\ln\left(1+\frac{T}{\left(\frac{\lambda q}{\epsilon}\right)^{2}+1}\right)}{\sqrt{T}}<\frac{16\epsilon}{\gamma L}\right\} }{2}\nonumber \\
 & \geq\frac{\left\lceil \frac{\lambda\Delta q}{2\gamma\epsilon^{2}}+\frac{\Delta^{2}}{16\gamma^{2}\epsilon^{2}}\right\rceil +\left\lceil \frac{\gamma^{2}L^{2}}{64\epsilon^{2}}\ln^{2}\left(\frac{\gamma L}{8\epsilon\sqrt{\left(\frac{\lambda q}{\epsilon}\right)^{2}+1}}\right)\right\rceil }{2},\label{eq:AdaGrad-lb-1d-Tstar-lb}
\end{align}
where the last step is due to (\ref{eq:AdaGrad-lb-1d-eps}) and Lemma
\ref{lem:AdaGrad-lb-1d-numeric-1}. Therefore, we obtain
\begin{align*}
T & \overset{(\ref{eq:AdaGrad-lb-1d-T}),(\ref{eq:AdaGrad-lb-1d-Tstar-lb})}{\geq}\Omega\left(\frac{\lambda\Delta}{\gamma\epsilon^{2}}+\frac{\Delta^{2}/\gamma^{2}+\gamma^{2}L^{2}\ln^{2}\left(\frac{\gamma L}{\lambda q+\epsilon}\right)}{\epsilon^{2}q}\right)\\
 & \overset{(\ref{eq:AdaGrad-lb-1d-q})}{=}\Omega\left(\frac{\lambda\Delta/\gamma+\Delta^{2}/\gamma^{2}+\gamma^{2}L^{2}\ln^{2}\left(\frac{\gamma L(1+\sigma^{\frac{\p}{\p-1}}\epsilon^{-\frac{\p}{\p-1}})}{\lambda+\epsilon(1+\sigma^{\frac{\p}{\p-1}}\epsilon^{-\frac{\p}{\p-1}})}\right)}{\epsilon^{2}}+\frac{\left(\Delta^{2}/\gamma^{2}+\gamma^{2}L^{2}\ln^{2}\left(\frac{\gamma L(1+\sigma^{\frac{\p}{\p-1}}\epsilon^{-\frac{\p}{\p-1}})}{\lambda+\epsilon(1+\sigma^{\frac{\p}{\p-1}}\epsilon^{-\frac{\p}{\p-1}})}\right)\right)\sigma^{\frac{\p}{\p-1}}}{\epsilon^{\frac{3\p-2}{\p-1}}}\right).
\end{align*}
\end{proof}

\subsection{Helpful Lemmas}

We provide three technical lemmas used in proving the algorithm-dependent
lower bound for $\AdaGrad$.

We first prove Lemma \ref{lem:AdaGrad-lb-1d-construction}, which
is essentially Lemma 20 of \citet{pmlr-v258-hubler25a} (see also
Lemma 15 of \citet{pmlr-v291-jiang25c}).
\begin{lem}
\label{lem:AdaGrad-lb-1d-construction}Given $\Delta>0$, $L>0$,
$0<\epsilon\leq\sqrt{2\Delta L}$, $\sx_{1}\in\R$, and a nonnegative
sequence $\left\{ \delta_{t}\right\} _{t=1}^{\infty}$, let 
\[
T_{\star}\defeq\inf\left\{ T\in\N:\Delta-\epsilon\sum_{t=1}^{T}\delta_{t}+\frac{L}{4}\sum_{t=1}^{T}\delta_{t}^{2}<\frac{\epsilon^{2}}{2L}\right\} ,
\]
and $\sy_{t}\defeq\sx_{1}+\sum_{s=1}^{t-1}\delta_{s},\forall t\in\left[T_{\star}\right]$,
then there exists a function $f:\R\to\R$ such that:
\begin{eqnarray*}
f(\sx_{1})-\inf_{\sx\in\R}f(\sx)\leq\Delta; & f\text{ is }L\text{-smooth}; & f'(\sy_{t})=-\epsilon,\forall t\in\left[T_{\star}\right].
\end{eqnarray*}
\end{lem}

\begin{proof}
For any $\delta\geq0$, let
\[
g_{\delta}(\sx)\defeq\begin{cases}
-\epsilon+L\sx & \sx\in\left[0,\delta/2\right]\\
-\epsilon+L\delta-L\sx & \sx\in\left(\delta/2,\delta\right]
\end{cases}.
\]
Now, we introduce 
\[
f'(\sx)\defeq\begin{cases}
-\epsilon & \sx<\sy_{1}\\
g_{\delta_{t}}(\sx-\sy_{t}) & \sx\in\left[\sy_{t},\sy_{t+1}\right),t\in\left[T_{\star}-1\right]\\
-\epsilon+L(\sx-\sy_{T_{\star}}) & \sx\geq\sy_{T_{\star}}
\end{cases},
\]
and 
\[
f(\sx)\defeq\Delta+\int_{\sy_{1}}^{\sx}f'(\sz)\d\sz.
\]
Note that $f$ is $L$-smooth and satisfies $f'(\sy_{t})=-\epsilon,\forall t\in\left[T_{\star}\right]$
by its definition. Thus, we only need to verify $f(\sx_{1})-\inf_{\sx\in\R}f(\sx)\leq\Delta$.
Note that $f(\sx_{1})=\Delta+\int_{\sy_{1}}^{\sx_{1}}f'(\sz)\d\sz=\Delta$
due to $\sy_{1}=\sx_{1}$, it remains to show $\inf_{\sx\in\R}f(\sx)\geq0$.

First, we can find that
\[
f(\sx)=\Delta+\epsilon(\sy_{1}-\sx),\forall\sx<\sy_{1}\Rightarrow\inf_{\sx<\sy_{1}}f(\sx)=f(\sy_{1}).
\]
Next, given $t\in\left[T_{\star}-1\right]$, we can find that
\[
\inf_{\sx\in\left[\sy_{t},\sy_{t+1}\right)}f(\sx)=\begin{cases}
\min\left\{ f(\sy_{t})-\frac{\epsilon^{2}}{2L},f(\sy_{t+1})\right\}  & \frac{\delta_{t}}{2}\geq\frac{\epsilon}{L}\\
f(\sy_{t+1}) & \frac{\delta_{t}}{2}<\frac{\epsilon}{L}
\end{cases}\geq\min\left\{ f(\sy_{t})-\frac{\epsilon^{2}}{2L},f(\sy_{t+1})\right\} .
\]
Finally, we know
\[
\inf_{\sx\geq\sy_{T_{\star}}}f(\sx)=f(\sy_{T_{\star}}+\epsilon/L)=f(\sy_{T_{\star}})-\frac{\epsilon^{2}}{2L}.
\]
The above three results together imply that
\[
\inf_{\sx\in\R}f(\sx)\geq\min_{t\in\left[T_{\star}\right]}f(\sy_{t})-\frac{\epsilon^{2}}{2L}.
\]
Now, we compute
\[
f(\sy_{t})=\Delta+\sum_{s=1}^{t-1}\int_{\sy_{s}}^{\sy_{s+1}}g_{\delta_{s}}(\sz-\sy_{s})\d\sz=\Delta-\epsilon\sum_{s=1}^{t-1}\delta_{s}+\frac{L}{4}\sum_{s=1}^{t-1}\delta_{s}^{2},\forall t\in\left[T_{\star}\right],
\]
which implies that, by the definition of $T_{\star}$,
\[
\min_{t\in\left[T_{\star}\right]}f(\sy_{t})-\frac{\epsilon^{2}}{2L}=\min_{t\in\left[T_{\star}\right]}\Delta-\epsilon\sum_{s=1}^{t-1}\delta_{s}+\frac{L}{4}\sum_{s=1}^{t-1}\delta_{s}^{2}-\frac{\epsilon^{2}}{2L}\geq0.
\]
\end{proof}

Next, Lemma \ref{lem:AdaGrad-lb-1d-numeric-1} provides a lower bound
for an important quantity used in the proof of Theorem \ref{thm:AdaGrad-lb-1d}.
\begin{lem}
\label{lem:AdaGrad-lb-1d-numeric-1}Let $c\in\left[3.92,3.93\right]$
denote the unique positive solution to $2c=(1+c)\ln(1+c)$, given
$A>0$ and $B\geq1$ satisfying $\ln\left(1+\frac{1}{B}\right)\geq A$
and $A\sqrt{B}\leq\frac{\ln c}{\sqrt{2c}}$, we have
\[
\inf\left\{ T\in\N:\frac{\ln\left(1+\frac{T}{B}\right)}{\sqrt{T}}<A\right\} \geq\left\lceil \frac{4}{A^{2}}\ln^{2}\left(\frac{2}{A\sqrt{B}}\right)\right\rceil .
\]
\end{lem}

\begin{proof}
Let $h(x)\defeq\frac{\ln\left(1+\frac{x}{B}\right)}{\sqrt{x}}$ for
$x>0$. We have
\[
h'(x)=\frac{B\left[\frac{2x}{B}-\left(1+\frac{x}{B}\right)\ln\left(1+\frac{x}{B}\right)\right]}{2x^{\frac{3}{2}}\left(B+x\right)}.
\]
We can find $h'(x)\geq0\Leftrightarrow\frac{2x}{B}-\left(1+\frac{x}{B}\right)\ln\left(1+\frac{x}{B}\right)\geq0\Leftrightarrow x\in\left(0,cB\right]$.
Therefore, $h(T)\geq h(1)=\ln\left(1+\frac{1}{B}\right)\geq A$ when
$T\in\left[\left\lfloor cB\right\rfloor \right]$, implying that
\begin{align*}
\inf\left\{ T\in\N:\frac{\ln\left(1+\frac{T}{B}\right)}{\sqrt{T}}<A\right\}  & =\inf\left\{ \left\lfloor cB\right\rfloor +1\leq T\in\N:\frac{\ln\left(1+\frac{T}{B}\right)}{\sqrt{T}}<A\right\} \\
 & \geq\inf\left\{ \left\lfloor cB\right\rfloor +1\leq T\in\N:\frac{\ln\left(\frac{T}{B}\right)}{\sqrt{T}}<A\right\} \\
 & =\inf\left\{ \left\lfloor cB\right\rfloor +1\leq T\in\N:-\frac{A\sqrt{T}}{2}\exp\left(-\frac{A\sqrt{T}}{2}\right)>-\frac{A\sqrt{B}}{2}\right\} .
\end{align*}
Now, we redefine $h(x)\defeq-x\exp(-x)$ for $x\geq0$. Note that
$h'(x)=\exp(-x)(x-1)\Rightarrow\min_{x\geq0}h(x)=h(1)=-\frac{1}{e}$.
Next, we observe that
\[
h\left(\frac{A\sqrt{\left\lfloor cB\right\rfloor +1}}{2}\right)\leq-\frac{A\sqrt{B}}{2}\Leftrightarrow\sqrt{\left\lfloor cB\right\rfloor +1}\geq\sqrt{B}\exp\left(\frac{A\sqrt{\left\lfloor cB\right\rfloor +1}}{2}\right),
\]
which holds due to 
\[
\sqrt{\left\lfloor cB\right\rfloor +1}\geq\sqrt{cB}\overset{A\sqrt{B}\leq\frac{\ln c}{\sqrt{2c}}}{\geq}\sqrt{B}\exp\left(\frac{A\sqrt{2cB}}{2}\right)\overset{cB\geq\left\lfloor cB\right\rfloor \geq1}{\geq}\sqrt{B}\exp\left(\frac{A\sqrt{\left\lfloor cB\right\rfloor +1}}{2}\right).
\]
Hence,
\[
\inf\left\{ \left\lfloor cB\right\rfloor +1\leq T\in\N:-\frac{A\sqrt{T}}{2}\exp\left(-\frac{A\sqrt{T}}{2}\right)>-\frac{A\sqrt{B}}{2}\right\} \geq\left\lceil T_{\text{root}}\right\rceil ,
\]
where $T_{\text{root}}\in\R$ is the unique solution of $h\left(\frac{A\sqrt{T}}{2}\right)=-\frac{A\sqrt{B}}{2}$
that guarantees $\frac{A\sqrt{T_{\text{root}}}}{2}>1$. More precisely,
let $W_{-1}$ be the Lambert $W$ function, we have
\[
\frac{A\sqrt{T_{\text{root}}}}{2}=-W_{-1}\left(-\frac{A\sqrt{B}}{2}\right)\Leftrightarrow T_{\text{root}}=\frac{4}{A^{2}}\left[-W_{-1}\left(-\frac{A\sqrt{B}}{2}\right)\right]^{2}.
\]
Finally, we apply the standard inequality $-W_{-1}(-x)\geq\ln\frac{1}{x}$
to conclude.
\end{proof}

Lastly, we prove Lemma \ref{lem:AdaGrad-lb-1d-numeric-2}, which can
help us further lower bound the algorithm-dependent lower bound.
\begin{lem}
\label{lem:AdaGrad-lb-1d-numeric-2}Given $A>0$, we have $\inf_{\eta>0}\frac{1}{\eta}+\eta\ln^{2}(A\eta)\geq\ln A.$
\end{lem}

\begin{proof}
We fix $\eta>0$ and define $h_{\eta}(x)\defeq\eta x^{2}-\left(1-2\eta\ln\eta\right)x+\eta\ln^{2}\eta+\frac{1}{\eta}$.
Note that $\eta>0$ and the discriminant of $h_{\eta}$ is $-4\eta\ln\eta-3\leq4/e-3<0$,
since $\min_{\eta>0}\eta\ln\eta=-1/e$. Therefore, $h_{\eta}(x)\geq0$
for all $x\in\R$. In particular, we have
\[
\frac{1}{\eta}+\eta\ln^{2}(A\eta)-\ln A=h_{\eta}(\ln A)\geq0,
\]
which implies the desired result.
\end{proof}

\section{Another Upper Bound for $\protect\AdaGradNorm$\label{sec:AdaGradNorm-ub}}

In this section, we provide another upper bound for $\AdaGradNorm$,
given in Theorem \ref{thm:AdaGradNorm-ub}. Unlike Theorem \ref{thm:main-AdaGradNorm-ub-bounded},
this bound does not require the objective function to be bounded.
However, it is only in the order of $\tilde{\O}(1/T^{\frac{3\p-4}{4\p}})$
as a trade-off (same as Theorem \ref{thm:AdaGrad-ub} for $\AdaGrad$),
which becomes vacuous when $\p\in\left(1,\nicefrac{4}{3}\right]$.

\subsection{Theorem and Its Proof}
\begin{thm}
\label{thm:AdaGradNorm-ub}Under Assumptions \ref{assu:lb}, \ref{assu:smooth},
\ref{assu:unbias}, and \ref{assu:heavy}, let $\Delta\triangleq f(\bx_{1})-f_{\star}$,
then for any $\gamma>0$ and $\lambda>0$, $\AdaGradNorm$ (Algorithm
\ref{alg:AdaGradNorm}) guarantees
\begin{align*}
\E\left[\frac{1}{T}\sum_{t=1}^{T}\left\Vert \nabla f(\bx_{t})\right\Vert _{2}\right]\leq & \O\left(\frac{\lambda+\frac{\Delta}{\gamma}+\gamma\left\Vert \bL\right\Vert _{\infty}\ln\res_{T}}{\sqrt{T}}+\frac{\left\Vert \bsig\right\Vert _{\p}\ln^{\frac{1}{\cp}+\frac{1}{2}}\res_{T}}{T^{\frac{\p-1}{\p}}}\right.\\
 & \left.\quad+\frac{\sqrt{\left(\frac{\Delta}{\gamma}+\gamma\left\Vert \bL\right\Vert _{\infty}\ln\res_{T}\right)\left\Vert \bsig\right\Vert _{\p}}}{T^{\frac{\p-1}{2\p}}}+\frac{\left\Vert \bsig\right\Vert _{\p}\ln^{\frac{1}{2\cp}+\frac{1}{4}}\res_{T}}{T^{\frac{3\p-4}{4\p}}}\right),
\end{align*}
where $\res_{T}=1+\frac{\sqrt{2}\left\Vert \bsig\right\Vert _{\p}T^{\frac{1}{\p}}+2\left\Vert \nabla f(\bx_{1})\right\Vert _{2}T^{\frac{1}{2}}+2\gamma\left\Vert \bL\right\Vert _{\infty}T^{\frac{3}{2}}}{\lambda}$
is introduced in Lemma \ref{lem:AdaGradNorm-ln-bound}.
\end{thm}

\begin{proof}
Equipped with Lemmas \ref{lem:AdaGradNorm-core} (choose $c\defeq\left\Vert \bsig\right\Vert _{\p}T^{\frac{1}{2}-\frac{1}{\cp}}/D_{T}^{\frac{1}{2}-\frac{1}{\cp}}$
for $D_{T}\defeq2\ln(1+\frac{\sqrt{2}\left\Vert \bsig\right\Vert _{\p}T^{\frac{1}{\p}}+\E\left[\sqrt{2u_{T}}\right]}{\lambda})$
when invoking it), \ref{lem:AdaGradNorm-general-ln}, \ref{lem:AdaGradNorm-v},
and \ref{lem:AdaGradNorm-ln-bound}, the proof of Theorem \ref{thm:AdaGradNorm-ub}
follows essentially the same way as proving Theorem \ref{thm:AdaGrad-ub},
which is omitted here to save space.
\end{proof}

\subsection{Helpful Lemmas}

This subsection provides all necessary lemmas to prove Theorem \ref{thm:AdaGradNorm-ub}.
The following four lemmas correspond to Lemmas \ref{lem:AdaGrad-core},
\ref{lem:AdaGrad-general-ln}, \ref{lem:AdaGrad-v}, and \ref{lem:AdaGrad-ln-bound}
under the $\ell_{2}$ geometry. Their proofs do not involve new techniques,
except that $w_{t}$ is now defined as follows
\[
w_{t}\defeq v_{t-1}+\left\Vert \nabla f(\bx_{t})\right\Vert _{2}^{2}+c^{2}\in\F_{t-1},\forall t\in\N,
\]
where $c\geq0$ can be an arbitrary constant and will be determined
in the proof of Theorem \ref{thm:AdaGradNorm-ub}.
\begin{lem}
\label{lem:AdaGradNorm-core}Under Assumptions \ref{assu:smooth},
\ref{assu:unbias}, and \ref{assu:heavy}, for any $c\geq0$ and $t\in\N$,
$\AdaGradNorm$ (Algorithm \ref{alg:AdaGradNorm}) guarantees
\[
\frac{\gamma}{2}\E\left[\frac{\left\Vert \nabla f(\bx_{t})\right\Vert _{2}^{2}}{\lambda+\sqrt{w_{t}}}\right]\leq\E\left[f(\bx_{t})-f(\bx_{t+1})\right]+\gamma\frac{\left\Vert \bsig\right\Vert _{\p}^{2}}{c}\left(\E\left[\frac{\left\Vert \bg_{t}\right\Vert _{2}^{2}}{\lambda^{2}+v_{t}}\right]\right)^{\frac{2}{\cp}}+\gamma\left(c+\frac{\gamma\left\Vert \bL\right\Vert _{\infty}}{2}\right)\E\left[\frac{\left\Vert \bg_{t}\right\Vert _{2}^{2}}{\lambda^{2}+v_{t}}\right].
\]
\end{lem}

\begin{proof}
We start with Assumption \ref{assu:smooth} and use the update rule
of $\AdaGradNorm$ to obtain
\begin{align*}
f(\bx_{t+1}) & \leq f(\bx_{t})+\left\langle \nabla f(\bx_{t}),\bx_{t+1}-\bx_{t}\right\rangle +\frac{\left\Vert \bx_{t+1}-\bx_{t}\right\Vert _{\bL}^{2}}{2}\\
 & =f(\bx_{t})-\gamma\left\langle \nabla f(\bx_{t}),\frac{\bg_{t}}{\lambda+\sqrt{v_{t}}}\right\rangle +\frac{\gamma^{2}\left\Vert \bg_{t}\right\Vert _{\bL}^{2}}{2(\lambda+\sqrt{v_{t}})^{2}}\\
 & \leq f(\bx_{t})-\gamma\left\langle \nabla f(\bx_{t}),\frac{\bg_{t}}{\lambda+\sqrt{v_{t}}}\right\rangle +\frac{\gamma^{2}\left\Vert \bL\right\Vert _{\infty}\left\Vert \bg_{t}\right\Vert _{2}^{2}}{2(\lambda^{2}+v_{t})}\\
 & =f(\bx_{t})-\gamma\left\langle \nabla f(\bx_{t}),\frac{\bg_{t}}{\lambda+\sqrt{w_{t}}}\right\rangle +\gamma\left\langle \nabla f(\bx_{t}),\frac{\bg_{t}}{\lambda+\sqrt{w_{t}}}-\frac{\bg_{t}}{\lambda+\sqrt{v_{t}}}\right\rangle +\frac{\gamma^{2}\left\Vert \bL\right\Vert _{\infty}\left\Vert \bg_{t}\right\Vert _{2}^{2}}{2(\lambda^{2}+v_{t})}.
\end{align*}
Take conditional expectations on both sides of the above inequality
and use 
\[
\E_{t-1}\left[\left\langle \nabla f(\bx_{t}),\frac{\bg_{t}}{\lambda+\sqrt{w_{t}}}\right\rangle \right]\overset{\bx_{t},w_{t}\in\F_{t-1}}{=}\left\langle \nabla f(\bx_{t}),\frac{\E_{t-1}\left[\bg_{t}\right]}{\lambda+\sqrt{w_{t}}}\right\rangle \overset{\text{Assumption }\ref{assu:unbias}}{=}\frac{\left\Vert \nabla f(\bx_{t})\right\Vert _{2}^{2}}{\lambda+\sqrt{w_{t}}}
\]
to obtain
\begin{align}
\E_{t-1}\left[f(\bx_{t+1})\right]\leq & f(\bx_{t})-\gamma\frac{\left\Vert \nabla f(\bx_{t})\right\Vert _{2}^{2}}{\lambda+\sqrt{w_{t}}}+\frac{\gamma^{2}\left\Vert \bL\right\Vert _{\infty}}{2}\E_{t-1}\left[\frac{\left\Vert \bg_{t}\right\Vert _{2}^{2}}{\lambda^{2}+v_{t}}\right]\nonumber \\
 & +\gamma\E_{t-1}\left[\left\langle \nabla f(\bx_{t}),\frac{\bg_{t}}{\lambda+\sqrt{w_{t}}}-\frac{\bg_{t}}{\lambda+\sqrt{v_{t}}}\right\rangle \right].\label{eq:AdaGradNorm-core-1}
\end{align}

Now, we can bound
\begin{align}
 & \E_{t-1}\left[\left\langle \nabla f(\bx_{t}),\frac{\bg_{t}}{\lambda+\sqrt{w_{t}}}-\frac{\bg_{t}}{\lambda+\sqrt{v_{t}}}\right\rangle \right]\nonumber \\
\leq & \E_{t-1}\left[\left\Vert \nabla f(\bx_{t})\right\Vert _{2}\left\Vert \bg_{t}\right\Vert _{2}\left|\frac{1}{\lambda+\sqrt{w_{t}}}-\frac{1}{\lambda+\sqrt{v_{t}}}\right|\right]\nonumber \\
\overset{(a)}{=} & \frac{\left\Vert \nabla f(\bx_{t})\right\Vert _{2}}{\lambda+\sqrt{w_{t}}}\E_{t-1}\left[\frac{\left\Vert \bg_{t}\right\Vert _{2}\left|\left\Vert \bg_{t}\right\Vert _{2}^{2}-\left\Vert \nabla f(\bx_{t})\right\Vert _{2}^{2}-c^{2}\right|}{(\lambda+\sqrt{v_{t}})(\sqrt{w_{t}}+\sqrt{v_{t}})}\right]\nonumber \\
\leq & \frac{\left\Vert \nabla f(\bx_{t})\right\Vert _{2}}{\lambda+\sqrt{w_{t}}}\left(\E_{t-1}\left[\frac{\left(\left\Vert \bg_{t}\right\Vert _{2}+\left\Vert \nabla f(\bx_{t})\right\Vert _{2}\right)\left\Vert \bxi_{t}\right\Vert _{2}\left\Vert \bg_{t}\right\Vert _{2}}{(\lambda+\sqrt{v_{t}})(\sqrt{w_{t}}+\sqrt{v_{t}})}\right]+\E_{t-1}\left[\frac{c^{2}\left\Vert \bg_{t}\right\Vert _{2}}{(\lambda+\sqrt{v_{t}})(\sqrt{w_{t}}+\sqrt{v_{t}})}\right]\right)\nonumber \\
\overset{(b)}{\leq} & \frac{\left\Vert \nabla f(\bx_{t})\right\Vert _{2}}{\lambda+\sqrt{w_{t}}}\left(\E_{t-1}\left[\frac{\left\Vert \bxi_{t}\right\Vert _{2}\left\Vert \bg_{t}\right\Vert _{2}}{\lambda+\sqrt{v_{t}}}\right]+\E_{t-1}\left[\frac{c\left\Vert \bg_{t}\right\Vert _{2}}{\lambda+\sqrt{v_{t}}}\right]\right)\nonumber \\
\overset{(c)}{\leq} & \frac{\left\Vert \nabla f(\bx_{t})\right\Vert _{2}^{2}}{2(\lambda+\sqrt{w_{t}})}+\frac{1}{\lambda+\sqrt{w_{t}}}\left(\left(\E_{t-1}\left[\frac{\left\Vert \bxi_{t}\right\Vert _{2}\left\Vert \bg_{t}\right\Vert _{2}}{\lambda+\sqrt{v_{t}}}\right]\right)^{2}+\left(\E_{t-1}\left[\frac{c\left\Vert \bg_{t}\right\Vert _{2}}{\lambda+\sqrt{v_{t}}}\right]\right)^{2}\right),\label{eq:AdaGradNorm-core-2}
\end{align}
where $(a)$ is by $\frac{\left\Vert \nabla f(\bx_{t})\right\Vert _{2}}{\lambda+\sqrt{w_{t}}}\in\F_{t-1}$,
$(b)$ is due to $\left\Vert \bg_{t}\right\Vert _{2}+\left\Vert \nabla f(\bx_{t})\right\Vert _{2}\leq\sqrt{w_{t}}+\sqrt{v_{t}}$
and $c\leq\sqrt{w_{t}}+\sqrt{v_{t}}$, and $(c)$ holds by AM-GM inequality,
i.e., $\left\Vert \nabla f(\bx_{t})\right\Vert _{2}X\leq\frac{\left\Vert \nabla f(\bx_{t})\right\Vert _{2}^{2}}{4}+X^{2}$
for $X=\E_{t-1}\left[\frac{\left\Vert \bxi_{t}\right\Vert _{2}\left\Vert \bg_{t}\right\Vert _{2}}{\lambda+\sqrt{v_{t}}}\right]$
and $\E_{t-1}\left[\frac{c\left\Vert \bg_{t}\right\Vert _{2}}{\lambda+\sqrt{v_{t}}}\right]$,
respectively. Next, we apply H\"{o}lder's inequality to get
\begin{align}
 & \left(\E_{t-1}\left[\frac{\left\Vert \bxi_{t}\right\Vert _{2}\left\Vert \bg_{t}\right\Vert _{2}}{\lambda+\sqrt{v_{t}}}\right]\right)^{2}\leq\left(\E_{t-1}\left[\left\Vert \bxi_{t}\right\Vert _{2}^{\p}\right]\right)^{\frac{2}{\p}}\left(\E_{t-1}\left[\frac{\left\Vert \bg_{t}\right\Vert _{2}^{\cp}}{(\lambda+\sqrt{v_{t}})^{\cp}}\right]\right)^{\frac{2}{\cp}}\nonumber \\
\overset{\left\Vert \cdot\right\Vert _{2}\leq\left\Vert \cdot\right\Vert _{\p},\text{Assumption }\ref{assu:heavy}}{\leq} & \left\Vert \bsig\right\Vert _{\p}^{2}\left(\E_{t-1}\left[\frac{\left\Vert \bg_{t}\right\Vert _{2}^{\cp}}{(\lambda+\sqrt{v_{t}})^{\cp}}\right]\right)^{\frac{2}{\cp}}\leq\left\Vert \bsig\right\Vert _{\p}^{2}\left(\E_{t-1}\left[\left(\frac{\left\Vert \bg_{t}\right\Vert _{2}^{2}}{\lambda^{2}+v_{t}}\right)^{\frac{\cp}{2}}\right]\right)^{\frac{2}{\cp}},\label{eq:AdaGradNorm-core-3}
\end{align}
and
\begin{equation}
\left(\E_{t-1}\left[\frac{c\left\Vert \bg_{t}\right\Vert _{2}}{\lambda+\sqrt{v_{t}}}\right]\right)^{2}\leq\E_{t-1}\left[\frac{c^{2}\left\Vert \bg_{t}\right\Vert _{2}^{2}}{(\lambda+\sqrt{v_{t}})^{2}}\right]\leq\E_{t-1}\left[\frac{c^{2}\left\Vert \bg_{t}\right\Vert _{2}^{2}}{\lambda^{2}+v_{t}}\right].\label{eq:AdaGradNorm-core-4}
\end{equation}
Plug (\ref{eq:AdaGradNorm-core-3}) and (\ref{eq:AdaGradNorm-core-4})
into (\ref{eq:AdaGradNorm-core-2}) to obtain
\begin{align}
 & \E_{t-1}\left[\left\langle \nabla f(\bx_{t}),\frac{\bg_{t}}{\lambda+\sqrt{w_{t}}}-\frac{\bg_{t}}{\lambda+\sqrt{v_{t}}}\right\rangle \right]\nonumber \\
\leq & \frac{\left\Vert \nabla f(\bx_{t})\right\Vert _{2}^{2}}{2(\lambda+\sqrt{w_{t}})}+\frac{\left\Vert \bsig\right\Vert _{\p}^{2}}{\lambda+\sqrt{w_{t}}}\left(\E_{t-1}\left[\left(\frac{\left\Vert \bg_{t}\right\Vert _{2}^{2}}{\lambda^{2}+v_{t}}\right)^{\frac{\cp}{2}}\right]\right)^{\frac{2}{\cp}}+\frac{c^{2}}{\lambda+\sqrt{w_{t}}}\E_{t-1}\left[\frac{\left\Vert \bg_{t}\right\Vert _{2}^{2}}{\lambda^{2}+v_{t}}\right]\nonumber \\
\leq & \frac{\left\Vert \nabla f(\bx_{t})\right\Vert _{2}^{2}}{2(\lambda+\sqrt{w_{t}})}+\frac{\left\Vert \bsig\right\Vert _{\p}^{2}}{c}\left(\E_{t-1}\left[\frac{\left\Vert \bg_{t}\right\Vert _{2}^{2}}{\lambda^{2}+v_{t}}\right]\right)^{\frac{2}{\cp}}+c\E_{t-1}\left[\frac{\left\Vert \bg_{t}\right\Vert _{2}^{2}}{\lambda^{2}+v_{t}}\right],\label{eq:AdaGradNorm-core-5}
\end{align}
where the last step is by $\lambda+\sqrt{w_{t}}\geq c$, $\frac{\left\Vert \bg_{t}\right\Vert _{2}^{2}}{\lambda^{2}+v_{t}}\leq1$,
and $\frac{\cp}{2}\geq1$.

We combine (\ref{eq:AdaGradNorm-core-1}) and (\ref{eq:AdaGradNorm-core-5})
to have
\[
\E_{t-1}\left[f(\bx_{t+1})\right]\leq f(\bx_{t})-\frac{\gamma\left\Vert \nabla f(\bx_{t})\right\Vert _{2}^{2}}{2(\lambda+\sqrt{w_{t}})}+\gamma\frac{\left\Vert \bsig\right\Vert _{\p}^{2}}{c}\left(\E_{t-1}\left[\frac{\left\Vert \bg_{t}\right\Vert _{2}^{2}}{\lambda^{2}+v_{t}}\right]\right)^{\frac{2}{\cp}}+\gamma\left(c+\frac{\gamma\left\Vert \bL\right\Vert _{\infty}}{2}\right)\E_{t-1}\left[\frac{\left\Vert \bg_{t}\right\Vert _{2}^{2}}{\lambda^{2}+v_{t}}\right].
\]
Taking expectations on both sides and rearranging terms, we know
\[
\frac{\gamma}{2}\E\left[\frac{\left\Vert \nabla f(\bx_{t})\right\Vert _{2}^{2}}{\lambda+\sqrt{w_{t}}}\right]\leq\E\left[f(\bx_{t})-f(\bx_{t+1})\right]+\gamma\frac{\left\Vert \bsig\right\Vert _{\p}^{2}}{c}\E\left[\left(\E_{t-1}\left[\frac{\left\Vert \bg_{t}\right\Vert _{2}^{2}}{\lambda^{2}+v_{t}}\right]\right)^{\frac{2}{\cp}}\right]+\gamma\left(c+\frac{\gamma\left\Vert \bL\right\Vert _{\infty}}{2}\right)\E\left[\frac{\left\Vert \bg_{t}\right\Vert _{2}^{2}}{\lambda^{2}+v_{t}}\right].
\]
Finally, noticing that $\frac{\cp}{2}\geq1$, we hence can invoke
H\"{o}lder's inequality again to have
\[
\E\left[\left(\E_{t-1}\left[\frac{\left\Vert \bg_{t}\right\Vert _{2}^{2}}{\lambda^{2}+v_{t}}\right]\right)^{\frac{2}{\cp}}\right]\leq\left(\E\left[\E_{t-1}\left[\frac{\left\Vert \bg_{t}\right\Vert _{2}^{2}}{\lambda^{2}+v_{t}}\right]\right]\right)^{\frac{2}{\cp}}=\left(\E\left[\frac{\left\Vert \bg_{t}\right\Vert _{2}^{2}}{\lambda^{2}+v_{t}}\right]\right)^{\frac{2}{\cp}},
\]
which leads us to the desired result.
\end{proof}

\begin{lem}
\label{lem:AdaGradNorm-general-ln}For any $T\in\N$ and $q\in\left[0,1\right]$,
$\AdaGradNorm$ (Algorithm \ref{alg:AdaGradNorm}) guarantees
\[
\sum_{t=1}^{T}\left(\E\left[\frac{\left\Vert \bg_{t}\right\Vert _{2}^{2}}{\lambda^{2}+v_{t}}\right]\right)^{q}\leq T^{1-q}\left(\E\left[\ln\left(1+\frac{v_{T}}{\lambda^{2}}\right)\right]\right)^{q}.
\]
\end{lem}

\begin{proof}
By the concavity of $x^{q}$ (since $q\in\left[0,1\right]$), we have
\[
\frac{1}{T}\sum_{t=1}^{T}\left(\E\left[\frac{\left\Vert \bg_{t}\right\Vert _{2}^{2}}{\lambda^{2}+v_{t}}\right]\right)^{q}\leq\left(\frac{1}{T}\sum_{t=1}^{T}\E\left[\frac{\left\Vert \bg_{t}\right\Vert _{2}^{2}}{\lambda^{2}+v_{t}}\right]\right)^{q},
\]
which implies that
\begin{align*}
\sum_{t=1}^{T}\left(\E\left[\frac{\left\Vert \bg_{t}\right\Vert _{2}^{2}}{\lambda^{2}+v_{t}}\right]\right)^{q} & \leq T^{1-q}\left(\E\left[\sum_{t=1}^{T}\frac{\left\Vert \bg_{t}\right\Vert _{2}^{2}}{\lambda^{2}+v_{t}}\right]\right)^{q}=T^{1-q}\left(\E\left[\sum_{t=1}^{T}1-\frac{\lambda^{2}+v_{t-1}}{\lambda^{2}+v_{t}}\right]\right)^{q}\\
 & \overset{(a)}{\leq}T^{1-q}\left(\E\left[\sum_{t=1}^{T}\ln\left(\frac{\lambda^{2}+v_{t}}{\lambda^{2}+v_{t-1}}\right)\right]\right)^{q}=T^{1-q}\left(\E\left[\ln\left(1+\frac{v_{T}}{\lambda^{2}}\right)\right]\right)^{q},
\end{align*}
where $(a)$ is due to $1-x^{-1}\leq\ln x,\forall x>0$.
\end{proof}

\begin{lem}[Restatement of Lemma \ref{lem:main-AdaGradNorm-v}]
\label{lem:AdaGradNorm-v}Under Assumption \ref{assu:heavy}, for
any $t\in\N$, $\AdaGradNorm$ (Algorithm \ref{alg:AdaGradNorm})
guarantees
\[
\E\left[\sqrt{v_{t}}\right]\leq\sqrt{2}\left\Vert \bsig\right\Vert _{\p}t^{\frac{1}{\p}}+\E\left[\sqrt{2u_{t}}\right],
\]
where $u_{t}\triangleq\sum_{s=1}^{t}\left\Vert \nabla f(\bx_{s})\right\Vert _{2}^{2},\forall t\in$$\N$.
\end{lem}

\begin{lem}
\label{lem:AdaGradNorm-ln-bound}Under Assumptions \ref{assu:smooth}
and \ref{assu:heavy}, $\AdaGradNorm$ (Algorithm \ref{alg:AdaGradNorm})
guarantees
\[
\E\left[\ln\left(1+\frac{v_{T}}{\lambda^{2}}\right)\right]\leq2\ln\left(1+\frac{\sqrt{2}\left\Vert \bsig\right\Vert _{\p}T^{\frac{1}{\p}}+\E\left[\sqrt{2u_{T}}\right]}{\lambda}\right)\leq2\ln\res_{T},
\]
where $\res_{T}\defeq1+\frac{\sqrt{2}\left\Vert \bsig\right\Vert _{\p}T^{\frac{1}{\p}}+2\left\Vert \nabla f(\bx_{1})\right\Vert _{2}T^{\frac{1}{2}}+2\gamma\left\Vert \bL\right\Vert _{\infty}T^{\frac{3}{2}}}{\lambda}.$
\end{lem}

\begin{proof}
Note that
\[
\E\left[\ln\left(1+\frac{v_{T}}{\lambda^{2}}\right)\right]=2\E\left[\ln\left(\sqrt{1+\frac{v_{T}}{\lambda^{2}}}\right)\right]\leq2\E\left[\ln\left(1+\frac{\sqrt{v_{T}}}{\lambda}\right)\right]\leq2\ln\left(1+\frac{\E\left[\sqrt{v_{T}}\right]}{\lambda}\right),
\]
where the last step is due to the concavity of $\ln x$. Next, we
invoke Lemma \ref{lem:AdaGradNorm-v} to obtain
\begin{equation}
\E\left[\ln\left(1+\frac{v_{T}}{\lambda^{2}}\right)\right]\leq2\ln\left(1+\frac{\sqrt{2}\left\Vert \bsig\right\Vert _{\p}T^{\frac{1}{\p}}+\E\left[\sqrt{2u_{T}}\right]}{\lambda}\right).\label{eq:AdaGradNorm-ln-bound-1}
\end{equation}

Moreover, under Assumption \ref{assu:smooth}, we have almost surely,
for any $t\in\left[T\right]$,
\begin{align*}
\left\Vert \nabla f(\bx_{t})-\nabla f(\bx_{1})\right\Vert _{\nicefrac{1}{\bL}} & \leq\left\Vert \bx_{t}-\bx_{1}\right\Vert _{\bL}\leq\sum_{s=1}^{t-1}\left\Vert \bx_{s+1}-\bx_{s}\right\Vert _{\bL}\\
 & =\sum_{s=1}^{t-1}\frac{\gamma\left\Vert \bg_{s}\right\Vert _{\bL}}{\lambda+\sqrt{v_{s}}}\leq\sum_{s=1}^{t-1}\gamma\sqrt{\left\Vert \bL\right\Vert _{\infty}}=\gamma\sqrt{\left\Vert \bL\right\Vert _{\infty}}(t-1)\\
\Rightarrow\left\Vert \nabla f(\bx_{t})-\nabla f(\bx_{1})\right\Vert  & \leq\gamma\left\Vert \bL\right\Vert _{\infty}(t-1).
\end{align*}
Hence, there is almost surely
\begin{align}
\sqrt{2u_{T}} & =\sqrt{2\sum_{t=1}^{T}\left\Vert \nabla f(\bx_{t})\right\Vert _{2}^{2}}\leq\sqrt{2\sum_{t=1}^{T}\left(\left\Vert \nabla f(\bx_{1})\right\Vert _{2}+\gamma\left\Vert \bL\right\Vert _{\infty}(t-1)\right)^{2}}\nonumber \\
 & \leq2\sqrt{\sum_{t=1}^{T}\left\Vert \nabla f(\bx_{1})\right\Vert _{2}^{2}+\sum_{t=1}^{T}\gamma^{2}\left\Vert \bL\right\Vert _{\infty}^{2}(t-1)^{2}}\nonumber \\
 & \leq2\left\Vert \nabla f(\bx_{1})\right\Vert _{2}T^{\frac{1}{2}}+2\gamma\left\Vert \bL\right\Vert _{\infty}T^{\frac{3}{2}}.\label{eq:AdaGradNorm-ln-bound-2}
\end{align}

Finally, we plug (\ref{eq:AdaGradNorm-ln-bound-2}) back into (\ref{eq:AdaGradNorm-ln-bound-1})
to have
\[
2\ln\left(1+\frac{\sqrt{2}\left\Vert \bsig\right\Vert _{\p}T^{\frac{1}{\p}}+\E\left[\sqrt{2u_{T}}\right]}{\lambda}\right)\leq2\ln\left(1+\frac{\sqrt{2}\left\Vert \bsig\right\Vert _{\p}T^{\frac{1}{\p}}+2\left\Vert \nabla f(\bx_{1})\right\Vert _{2}T^{\frac{1}{2}}+2\gamma\left\Vert \bL\right\Vert _{\infty}T^{\frac{3}{2}}}{\lambda}\right)=2\ln\res_{T}.
\]
\end{proof}

\end{document}